\documentclass{article}

\usepackage{slantsc}
\usepackage[utf8]{inputenc}
\usepackage[T1]{fontenc}
\usepackage[margin=1.6in]{geometry}
\usepackage{titlesec} % titles modifications 
\usepackage{needspace}
\usepackage{float}
\usepackage{enumerate}
\usepackage{authblk} % affiliations
\usepackage{microtype}
\usepackage{environ}

\usepackage{amsthm, amsmath, amssymb}
\usepackage[full, bold]{complexity}
\usepackage{mathtools}
\usepackage[ruled, linesnumbered]{algorithm2e}
\usepackage{thmtools}
\usepackage{thm-restate}
\usepackage{hyperref}
\usepackage[sort]{cleveref}
\usepackage{graphicx}
\usepackage{tabularx}
\usepackage{eucal}
\usepackage{lineno}
\usepackage{xcolor}

\def\cqedsymbol{\ifmmode$\lrcorner$\else{\unskip\nobreak\hfil
\penalty50\hskip1em\null\nobreak\hfil$\lrcorner$
\parfillskip=0pt\finalhyphendemerits=0\endgraf}\fi}

\renewenvironment{abstract}
{\small\vspace{-1em}
\begin{center}
\bfseries\abstractname\vspace{-.5em}\vspace{0pt}
\end{center}
\list{}{
\setlength{\leftmargin}{0.6in}%
\setlength{\rightmargin}{\leftmargin}}%
\item\relax}
{\endlist}

\declaretheorem[parent=section, name=Theorem, style=plain]{theorem}
\declaretheorem[sibling=theorem, name=Lemma, style=plain]{lemma}
\declaretheorem[sibling=theorem, name=Corollary, style=plain]{corollary}
\declaretheorem[sibling=theorem, name=Proposition, style=plain]{proposition}
\declaretheorem[sibling=theorem, name=Question, style=plain]{question}

\theoremstyle{remark} % Apply a distinct style (similar to proof)

\newcommand{\N}{\ensuremath{\mathcal{N}}}

\renewcommand{\H}{\mathcal{H}}
\renewcommand{\G}{\mathcal{G}}

\newcommand{\T}{\ensuremath{\mathcal{T}}}
\renewcommand{\E}{\ensuremath{\mathcal{E}}}

\newcommand{\algoA}{\ensuremath{\mathtt{A}}}
\newcommand{\algoB}{\ensuremath{\mathtt{B}}}

\newcommand{\minred}{\operatorname{\mathsf{MinRed}}}
\newcommand{\maxirr}{\operatorname{\mathsf{MaxIrr}}}

\newcommand{\priv}{\operatorname{\mathsf{priv}}}

\newcommand{\inc}{\operatorname{\mathsf{inc}}}
\newcommand{\parent}{\operatorname{\mathsf{parent}}}
\newcommand{\children}{\operatorname{\mathsf{children}}}

\newcommand{\red}{\operatorname{\mathsf{red}}}

\newcommand{\problemdot}{\(\cdot\allowbreak{}\)}

\newcommand{\graphmirrenum}{\textsc{Graph\allowbreak{}Mirr}\problemdot\textsc{Enum}}
\newcommand{\hypmirrenum}{\textsc{HypMirr}\problemdot\textsc{Enum}}
\newcommand{\graphmredenum}{\textsc{GraphMred}\problemdot\textsc{Enum}}
\newcommand{\hypmredenum}{\textsc{HypMred}\problemdot\textsc{Enum}}
\newclass{\TotalP}{TotalP}
\newclass{\DelayP}{DelayP}
\newclass{\IncP}{IncP}
\newclass{\TotalQP}{TotalQP}
\newclass{\DelayQP}{DelayQP}
\newclass{\IncQP}{IncQP}
\newclass{\TotalXP}{TotalXP}
\newclass{\IncXP}{IncXP}
\newclass{\DelayXP}{DelayXP}
\newclass{\TotalFPT}{TotalFPT}
\newclass{\IncFPT}{IncFPT}
\newclass{\DelayFPT}{DelayFPT}

\makeatletter
\newcommand{\succprec}{\mathrel{\mathpalette\succ@prec{\succ\prec}}}
\newcommand{\precsucc}{\mathrel{\mathpalette\succ@prec{\prec\succ}}}

\newcommand{\succ@prec}[2]{\succ@@prec#1#2}
\newcommand{\succ@@prec}[3]{%
  \vcenter{\m@th\offinterlineskip
    \sbox\z@{$#1#3$}%
    \hbox{$#1#2$}\kern-0.4\ht\z@\box\z@
  }%
}
\makeatother

\makeatletter
\newcommand{\problemtitle}[1]{\gdef\@problemtitle{#1}}% 
\newcommand{\problemparameter}[1]{\gdef\@problemparameter{#1}}% Store problem title
\newcommand{\probleminput}[1]{\gdef\@probleminput{#1}}% Store problem input
\newcommand{\problemquestion}[1]{\gdef\@problemquestion{#1}}% Store problem question
\NewEnviron{enumproblem}{
	\problemtitle{}\problemparameter{}\probleminput{}\problemquestion{}% Default input is empty
	\BODY% Parse input
	\par\addvspace{.5\baselineskip}
	\noindent
	\begin{tabularx}{\textwidth}{@{\hspace{\parindent}} l X c}
		\multicolumn{2}{@{\hspace{\parindent}}l}{\@problemtitle} \\% Title
		\textbf{Input:} & \@probleminput \\% Input
		\textbf{Output:} & \@problemquestion% Output
	\end{tabularx}
	\par\addvspace{.5\baselineskip}
}

\NewEnviron{decproblem}{
	\problemtitle{}\problemparameter{}\probleminput{}\problemquestion{}% Default input is empty
	\BODY% Parse input
	\par\addvspace{.5\baselineskip}
	\noindent
	\begin{tabularx}{\textwidth}{@{\hspace{\parindent}} l X c}
		\multicolumn{2}{@{\hspace{\parindent}}l}{\@problemtitle} \\% Title
		\textbf{Input:} & \@probleminput \\% Input
		\textbf{Question:} & \@problemquestion% Output
	\end{tabularx}
	\par\addvspace{.5\baselineskip}
}

\title{Generating minimal redundant and maximal irredundant sets in incidence graphs%
\thanks{The first, third, and last authors have been supported by the ANR project PARADUAL (ANR-24-CE48-0610-01).}
}

\author[1]{Emanuel Castelo}
\author[1]{Jérémie Chalopin}
\author[1]{Oscar Defrain}
\author[1]{Simon Vilmin}

\affil[1]{Aix-Marseille Université, CNRS, LIS, Marseille, France.}
\newif\iflongversion
\longversiontrue
\newcommand{\iflongelse}[2]{\iflongversion{#1}\else{#2}\fi}

\makeatletter
\renewcommand\subparagraph{%
  \@startsection{subparagraph}{5}{\z@}%
    {3.25ex \@plus1ex \@minus.2ex}%  beforeskip (same as paragraph)
    {-1em}%                         afterskip (run-in, same as paragraph)
    {\normalfont\normalsize\bfseries}% style (same as paragraph)
}
\makeatother

\begin{document}

\maketitle

\vspace{-.3cm}

\begin{abstract}
It has been proved by Boros and Makino that there is no output-polynomial-time algorithm enumerating the minimal redundant sets or the maximal irredundant sets of a hypergraph, unless \(\P = \NP\).
The same question was left open for graphs, with only a few tractable cases known to date.
In this paper, we focus on graph classes that capture incidence relations such as bipartite, co-bipartite, and split graphs.
Concerning maximal irredundant sets, we show that the problem on co-bipartite graphs is as hard as in general graphs and tractable in split and strongly orderable graphs, the latter being a generalization of chordal bipartite graphs.
As formeta minimal redundant sets enumeration, we first show that the problem is intractable in split and co-bipartite graphs, answering the aforementioned open question, and that it is tractable on $(C_3,C_5,C_6,C_8)$-free graphs, a class of graphs incomparable to strongly orderable graphs, and which also generalizes chordal bipartite graphs.

\vspace{0.5em}
\end{abstract}

\section{Introduction}

\emph{Irredundant sets} in hypergraphs are those sets of vertices $I$ for which every element $x\in I$ has a private edge, i.e., an edge whose intersection with \(I\) is exactly \(\{x\}\).
Equivalently, they are often presented as sets of vertices that are minimal with respect to the set of hyperedges they intersect.
Sets of vertices that are not irredundant are called \emph{redundant}.

In graphs, \emph{irredundant sets} are defined as the sets of vertices having private neighbors, that is, neighbors that are not adjacent to other vertices in the set, allowing a vertex to be its own private neighbor.
\emph{Redundant sets} are defined analogously as the subsets of vertices that are not irredundant.
Redundant and irredundant sets of a graph should not be confused with the redundant and irredundant sets of the underlying hypergraph where each edge defines a hyperedge.
Rather, they correspond to the redundant and irredundant sets of the hypergraph of closed-neighborhoods of the graph, in an analogue of what is domination to transversality; we refer to Section~\ref{section:preliminaries} for a more in-depth presentation of these notions together with some examples.

Irredundancy in graphs was first introduced by Cockayne, Hedetniemi, and Miller in~\cite{cockayne1978properties} for its links with independence and domination, where it was in particular noted that every maximal independent set is a minimal dominating set, which in turn is a maximal irredundant set.
Due to its relation to independence and domination, irredundancy in graphs has been extensively studied~\cite{bollobas1979graph,cockayne1981contributions,bollobas1984irredundance,favaron1986stability,bertossi1988total,golumbic1993irredundancy,jacobson1990chordal}.
Let us mention that computing the minimum or maximum size of a maximal irredundant set is known to be \NP-complete~\cite{laskar1983domination,hedetniemi1985irredundance,fellows1994private}.

The generalization of irredundancy to hypergraphs has been considered for algorithmic purposes in the contexts of minimal domination or transversality \cite{kante2015edge,bartier2024hypergraph}, for it is known that minimal dominating sets (resp.~minimal tranversals) are those dominating sets (resp.~tranversals) which are irredundant.

Despite the interest for irredundant sets, little is known regarding their enumeration and that of their dual, redundant sets. Here the goal is, given a graph or a hypergraph, to list without repetition all its maximal irredundant sets or minimal redundant sets. 
In the following, let us denote by \graphmirrenum{}, \graphmredenum{}, \hypmirrenum{}, and \hypmredenum{} these problems (see Section~\ref{section:preliminaries} for the formal definition). 
Since the number of minimal redundant or maximal irredundant sets can be exponential in the size of the input (hyper)graph, evaluating the complexity of an enumeration algorithm solving either of these problems using the input size only is not a relevant efficiency measure.
Instead algorithms are analyzed using \emph{output-sensitive complexity} that estimates the running time of an algorithm in terms of both its input and output sizes.
An algorithm runs in \emph{output-polynomial time} if its execution time is polynomially bounded by the combined sizes of the input and the output.
This notion can be further detailed to analyze the delay between two solutions. 
Namely, we say that an algorithm runs in \emph{incremental-polynomial time} if it outputs the $i$-th solution after a time being polynomial in the input size plus $i$. 
If the delay between before the first output, between two consecutive outputs, and after the last output is polynomial in the input size only, then the algorithm is said to run with \emph{polynomial delay}.
We redirect the reader to \cite{strozecki2019enumeration} for further details on enumeration complexity.

We now review the state of the art concerning the four above-mentioned problems.
%and refer to Section~\ref{section:preliminaries} for their definition and an overview of enumeration complexity.
The complexity status of \hypmirrenum{} and \hypmredenum{} was first stated as an open problem by Uno during the Lorentz Workshop on ``Enumeration Algorithms using Structure'' held in 2015 \cite{bodlaender2015open}.
Shortly after, Boros and Makino announced the following as a negative answer for both problems \cite{boros2016wepa,boros2024generating}, where the \emph{dimension} of a hypergraph is the maximum size of its edges.

\begin{theorem}[{Reformulation of \cite[Theorems~1 \& 2]{boros2024generating}}]\label{thm:boros-hardness}
    There is no output-polynomial-time algorithm for \hypmirrenum{} and \hypmredenum{} unless $\P\neq \NP$, even when restricted to hypergraphs of dimension at most three.
\end{theorem}

Let us point that the original formulations in \cite{bodlaender2015open,boros2016wepa,boros2024generating} are stated in the context of redundant and irredundant subhypergraphs,\footnote{A set of hyperedges (or subhypergraph) is called \emph{irredundant} if each of its hyperedge contains a private vertex, i.e., a vertex that is only contained in this hyperedge within the set; it is called \emph{redundant} otherwise.} 
which is equivalent to ours up to considering the transposed\footnote{%
The \emph{transposed hypergraph} is obtained by considering each edge as a vertex, and each vertex as an edge, while keeping their mutual incidences. In terms of its incidence bipartite graph, it amounts to ``swap'' its two sides. In particular, the dimension of the hypergraph becomes the degree of its transposed hypergraph, and its degree becomes its dimension.%
} hypergraph. 
% In particular, the original statement is stated for the maximum \emph{degree}, i.e., the maximum number of edges a vertex intersects.
% Actually, in our context, Theorem~\ref{thm:boros-hardness} even holds for hypergraphs of dimension at most three, where the \emph{dimension} is the maximum size of an edge.
The above statement is thus a reformulation to our context of vertex subsets, as will be the following summary of their result.
In addition to Theorem~\ref{thm:boros-hardness}, the authors in \cite{boros2016wepa,boros2024generating} show that \hypmredenum{} and \hypmirrenum{} respectively admit polynomial and output-polynomial times algorithms in hypergraphs of bounded \emph{degree}\, which is the number of edges a vertex intersects.
When restricted to hypergraphs of dimension at most two, they show that \hypmredenum{} can be solved in polynomial time, while \hypmirrenum{} is equivalent to hypergraph dualization, arguably one of the most important open problem in enumeration for which no better than incremental-quasi-polynomial time is known~\cite{eiter1995identifying,fredman1996complexity,eiter2008computational}.
Let us stress again the fact that the results for hypergraphs of dimension at most two do not apply to \graphmredenum{} nor \graphmirrenum{}, as (ir)redundancy in graphs correspond to (ir)redundancy in the hypergraph of closed-neighborhoods of these graphs, for which the dimension is not bounded.

Later in~\cite{conte2019maximal}, the authors study \hypmirrenum{} under the prism of parameterized complexity and in particular derive a polynomial-delay algorithm on hypergraphs (hence graphs) of bounded degeneracy.

In \cite{blind2021locally}, it is shown that \graphmirrenum{} admits an algorithm running with linear delay after polynomial-time preprocessing and polynomial space when restricted to circular-permutation and permutation graphs.

Finally, let us briefly mention that irredundant sets enumeration has triggered some interest in the context of input-sensitive enumeration \cite{binkele2011breaking,golovach2019enumerationchordal,golovach2019enumerationclaw}, where the goal is to devise exponential time algorithms minimizing the base of the exponent.

In this work, we continue the research direction that was initiated in~\cite{blind2021locally}, as a restriction of the original question asked in \cite{bodlaender2015open}, and wonder whether \graphmredenum{} and \graphmirrenum{} are tractable.
Namely, we wish to answer the following question.

\begin{question}\label{qu:output-poly-for-graphs}
    Can \graphmredenum{} and \graphmirrenum{} be solved in output-polynomial time on general graphs?
\end{question}

We note that the same question is posed as an open direction in \cite{conte2019maximal}.

In the light of \Cref{thm:boros-hardness}, and driven by the intuition that the restrictions of \hypmredenum{} and \hypmirrenum{} to graphs should  still be intractable,\footnote{In what would be an analogue of transversals to dominating sets \cite{kante2014enumeration}.} we address \Cref{qu:output-poly-for-graphs} focusing on graph classes that capture incidence relations of hypergraphs such as bipartite, co-bipartite, and split graphs; we refer to Section~\ref{section:preliminaries} for their relation to hypergraphs.

% 
% We end this section with a summary of the results obtained in this paper, together with a brief description of the methods employed.
% As before, the reader is referred to~\Cref{section:preliminaries} for a precise definition of the terminology.
% 

First, we focus on maximal irredundant sets. 
\iflongelse{%
    On split graphs, it was already claimed by Uno in \cite{bodlaender2015open} that maximal irredundant sets are in bijection with minimal dominating sets.
    We include a proof of this statement for completeness.
    We thus derive the following using the algorithm from \cite{kante2014enumeration} for minimal dominating sets enumeration in this graph class.%
}
{
    On split graphs, it was already claimed by Uno in \cite{bodlaender2015open} that maximal irredundant sets are in bijection with minimal dominating sets, from which we derive the following using the algorithm in \cite{kante2014enumeration}; a proof of this statement is included in appendix.%
}

\begin{restatable}[{\cite{bodlaender2015open}}]{theorem}{thmMIRRsplit}
    \label{thm:MIRR:split}
    There is a bijection between maximal irredundant sets and minimal dominating sets in split graphs. Consequently, \graphmirrenum{} can be solved with linear delay and space on this class.
\end{restatable}

% 
% The fact that minimal dominating sets are maximal irredundant sets (see~\Cref{section:preliminaries}) is a well-known result.
% Additionally, we prove that when restricted to split graphs the maximal irredundant sets are dominating, and thus concluding the bijection between both combinatorial structures and allowing us to derive~\Cref{thm:MIRR:split} by using the algorithm from~\cite{kante2014enumeration}.
% 

As for co-bipartite graphs, we obtain the following negative result, which shows the problem to be challenging in this class.

\begin{restatable}{theorem}{thmMIRRcobip}
    \label{thm:MIRR:cobip}
    There is an output-polynomial-time algorithm solving \graphmirrenum{} on general graphs if and only if there is one solving the same problem on co-bipartite graphs.
\end{restatable}

To obtain~\Cref{thm:MIRR:cobip} we consider the hypergraph of closed neighborhoods of the input graph.
We then use the symmetry of its corresponding co-bipartite incidence graph to argue that, besides a polynomial number of solutions of bounded size, the problem amounts to generating all maximal irredundant sets of the input graph twice, and thus obtain the aforementioned result.

Finally, we obtain the following tractable result on strongly orderable graphs, which generalize chordal bipartite graphs \cite{dragan2000strongly}.

\begin{restatable}{theorem}{thmMIRRstronglyorderable}
    \label{thm:MIRR:strongly_orderable}
    There is a polynomial-delay and polynomial-space algorithm solving \graphmirrenum{} on strongly orderable graphs.
\end{restatable}

We devise the algorithm of~\Cref{thm:MIRR:strongly_orderable} by using ordered generation in the trace of the closed-neighborhood hypergraph.
Our method is based on earlier work of~\cite{conte2019maximal}, where the algorithm aims to break any redundancies obtained from adding a new vertex into the set by using structural properties of an elimination ordering.

% \odtodo[inline]{A rough description of the techniques used to prove these 3 statements. Could be split (below each statement) or as a big paragraph here. Incidence of a hypergraph (very rough description + ref to prelim) + ordered generation on the hypergraph of neighborhoods but using the trace.}

We turn to minimal redundant sets and obtain the following definite answer to \Cref{qu:output-poly-for-graphs} for \graphmredenum{}.

\begin{restatable}{theorem}{thmMREDsplitcobip}
    \label{thm:MRED:split-cobip}
    There is no output-polynomial-time algorithm for \graphmredenum{} unless $\P\neq \NP$, even when restricted to split and co-bipartite graphs.
\end{restatable}

Both results are obtained by reducing from \hypmredenum{} and arguing that, except for a polynomial number of minimal redundant sets, the remaining solutions of the constructed instance of \graphmredenum{} are in an one-to-one correspondence with the solutions of the hypergraph.
To argue the existence of no more than a polynomial number of solutions to discard, we exploit the structure of split graphs and, for the case of co-bipartite graphs, the fact that \hypmredenum{} remains intractable when the hypergraph has dimension equals three.

On the other hand, we obtain the following positive results.
Let us note that $(C_3,C_5,\allowbreak C_6,C_8)$-free graphs contain chordal bipartite graphs, which are $(C_3, C_{\geq 5})$-free~\cite{golumbic1978perfect}. 

\begin{restatable}{theorem}{thmMREDceightfree}
    \label{thm:MIRR:C3C5C6C8}
    There is a polynomial-delay algorithm that solves
    \graphmredenum{} on $(C_3,C_5,C_6,C_8)$-free graphs, and an incremental quasi-polynomial-time algorithm that solves \graphmredenum{} on $(C_3,C_5,C_6)$-free graphs.
\end{restatable}

These algorithms rely on a characterization of the minimal redundant sets of \((C_{3}, C_{5},\allowbreak C_{6})\)-free graphs.
In particular, we show that the number of redundant vertices in a minimal redundant set is bounded for these classes of graphs and that, except for the case of a unique redundant vertex, the set either has bounded size or corresponds to the closed neighborhood of some vertex in the set.
The remaining case where the set has a single redundant vertex is then reduced to instances of red-blue domination
%to instances of computing minimal red-blue dominating sets.
% Hence, we conclude both Theorems.
for which algorithms running in the specified times are known on these classes.

% \odtodo[inline]{A rough description of the techniques used to prove these 3 statements. Incidence of a hypergraph + strategy of partitioning solutions depending on the number of redundant vertices in a minimal redundant set.}

\subparagraph{Organization.} The rest of the paper is organized as follows.
We introduce concepts in \Cref{section:preliminaries}.
We focus on maximal irredundant sets in \Cref{section:MIRR} that we break into subsections according to the graph classes we consider.
We then turn on minimal redundant sets in \Cref{section:MRED} that we organize analogously.
Lastly, future directions are discussed in \Cref{sec:conclusion}.
\iflongelse{}{Due to space constraints, all proofs are omitted and the reader is referred to the long version of this manuscript (in appendix) for more details. For the same reasons, only the case of co-bipartite graphs in \Cref{thm:MRED:split-cobip} is addressed in this extended abstract.}

\section{Preliminaries}
\label{section:preliminaries}

% For each \(x \in V(G)\) the set \(B(x, d) \coloneqq \{y \in V(G) \mid \dist(x, y) \leq d\}\) is the \emph{ball of radius \(d\) centered at \(x\)}. \odtodo{we may want to use $N^2[x]$ as we used similar $N^2(x)$ in previous sections and $B$ is taken for bip of incidence}

In this paper, we consider undirected finite simple graphs.
We assume the reader is familiar with standard terminology from graph and hypergraph theories, and refer to~\cite{berge1984hypergraphs, diestel2005graph} for the notations not defined below. 
Given a vertex $v$ and an integer $d$, by $N^d(v)$ and $N^d[v]$ we respectively mean the set of vertices lying at distance exactly $d$ and at most $d$ from $v$.

\subparagraph{Strongly orderable graphs.} Given a graph \(G\), a pair of vertices \(x\) and \(y\) of \(G\) are \emph{comparable} if either \(N(x) \setminus \{y\} \subseteq N(y) \setminus \{x\}\) or \(N(y) \setminus \{x\} \subseteq N(x) \setminus \{y\}\).
A vertex \(x\) is \emph{quasi-simple} if the vertices in \(N(x)\) are pairwise comparable.
In Figure~\ref{fig:quasi-simple-vertex}, we give an example of a graph admitting a quasi-simple vertex.
Lastly, a \emph{quasi-simple elimination ordering} is an ordering \(v_{1}, \dots, v_{n}\) of \(V(G)\) such that for every \(i \in [n]\) it follows that \(v_{i}\) is quasi-simple in the graph \(G[\{v_{1}, \dots, v_{i}\}]\) induced by the $i$ first vertices.
Graphs that admit a quasi-simple elimination ordering are precisely the \emph{strongly orderable graphs}, which were introduced as a generalization of both chordal bipartite and strongly chordal graphs in~\cite{dragan2000strongly}.

\begin{figure}
    \centering
    \includegraphics[scale=0.85]{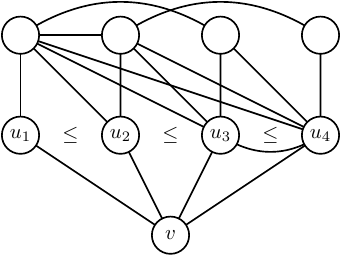}
    \caption{A graph $G$ in which the vertex $v$ is quasi-simple. 
    Within its neighborhood, we have $u_1 \leq u_2 \leq u_3 \leq u_4$ where $\leq$ indicates comparability.}
    \label{fig:quasi-simple-vertex}
\end{figure}

\subparagraph{Hypergraphs.} Let \(\H\) be a hypergraph.
Then, the vertices and edges of \(\H\) are denoted by \(V(\H)\) and \(\E(\H)\), respectively.
We call \(\inc_{\H}(x) \coloneqq \{E \in \E(\H) \mid x \in E\}\) the set of hyperedges containing \(x\).
The subscript is omitted when the hypergraph is clear from context.
If \(G\) is a graph, then its \emph{closed-neighborhood hypergraph}, denoted $\N(G)$, is defined on the same vertex set and hyperedge set \(\{N[x] \mid x \in V(G)\}\).

Let \(E_{1}, \dots, E_{m}\) be an ordering of \(\E(\H)\).
The \emph{transposed hypergraph} of $\H$ is the hypergraph $\H^t$ with $V(\H^{t}) \coloneqq \{e_1, \dots, e_m\}$ and $\E(\H^{t}) \coloneqq \{F_1, \dots, F_n\}$ where $F_i = \{e_j \in V(\H^{t}) \mid v_i \in E_j\}$.
Moreover, given a set \(S \subseteq V\), the \emph{trace} of \(\H\) with respect to \(S\) is the hypergraph \(\H_{S}\) where \(V(\H_{S}) = S\) and \(\E(\H_{S}) \coloneqq \{E \cap S \mid E \in \E(\H),\ E \cap S \neq \emptyset\}\).

\subparagraph{Incidence graphs.} Let $\H$ be a hypergraph with $V(\H) = \{v_1, \dots, v_n\}$ and $\E(\H) = \{E_1, \dots,\allowbreak E_m\}$.
The \emph{incidence bipartite graph} of $\H$ is the bipartite graph $B(\H)$ with parts $V = V(\H)$ and $U \coloneqq \{e_1, \dots, e_m\}$,  and edge set  $\{v_i e_j \mid v_i \in E_j\}$.
Similarly, the \emph{incidence co-bipartite graph} of $\H$ is the co-bipartite graph $C(\H)$ obtained from $B(\H)$ by turning $V$ and $U$ into cliques.
Finally, we can define two \emph{incidence split graphs} for $\H$ depending on which of $V$ and $U$ is the independent set.
We define $S_1(\H)$ to be the split graph obtained from $B(\H)$ by making $U$ a clique, and $S_2(\H)$ to be the split graph where $V$ is the clique.
We give in Figure~\ref{fig:incidence-graphs} an example illustrating all four definitions. 

\begin{figure}[hb!]
    \centering
    \includegraphics[width=0.95\linewidth]{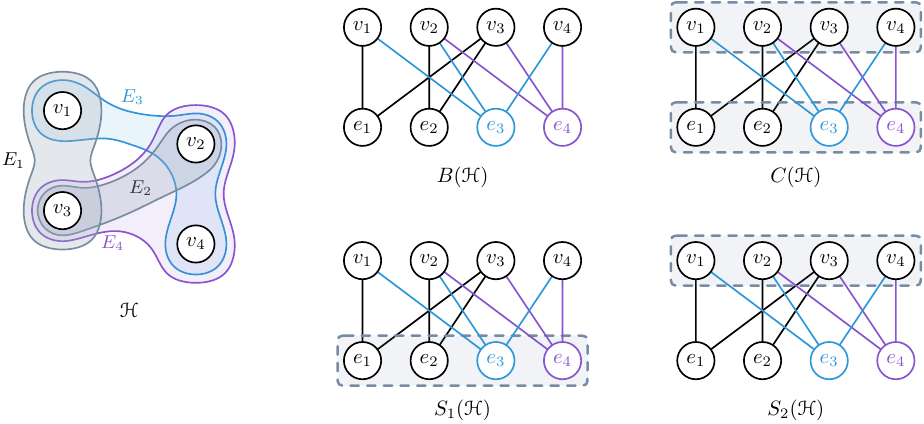}
    \caption{An hypergraph $\H$ and its incidence graphs $B(\H)$, $C(\H)$, $S_1(\H)$ and $S_2(\H)$.
    In graphs, cliques are indicated by dotted grey zones.}
    \label{fig:incidence-graphs}
\end{figure}

\subparagraph{Domination.} Let \(G\) be a graph.
We say that a subset of vertices $D$ is a \emph{dominating set} of $G$ if $V(G)=N[D]$.
Given a bipartition \((A, B)\) of the vertices, we say that a subset $D\subseteq A$ \emph{dominates} $B$ if \(B \subseteq N(D)\).
In this latter context, these sets are usually referred to as \emph{red-blue dominating sets}, where elements of $A$ are assumed to be colored red, and those in $B$ to be blue.
Generating every minimal red-blue dominating set is known to be solvable with polynomial delay and space when restricted to chordal bipartite graphs~\cite{golovach2016enumerating, castelo2025enumerating}.
It can also be solved with polynomial delay in \((C_{6}, C_{8})\)-free bipartite graphs~\cite{kante2018enumerating} as long as red and blue vertices belong to distinct parts of the graph.
In general, the problem is equivalent to hypergraph dualization and can thus be solved in incremental-quasi-polynomial time; see e.g.~\cite{eiter2008computational,golovach2016enumerating}.
% generating every minimal transversal of the hypergraph defined as \(\H \coloneqq \{N(x) \cap A \mid x \in B\}\).
% Lastly, given \(S \subseteq V(G)\) we define \(\minrb(G, S) \coloneqq \{D \subseteq V(G) \setminus S \mid D \text{ is a minimal red-blue dominating set of } S\}\) as the set of every minimal red-blue dominating set covering \(S\).

\subparagraph{(Ir)redundancy.}
Let $\H$ be a hypergraph, $I \subseteq V(\H)$, and $x\in I$.
We call \emph{private edge of $x$ with respect to $I$} the edges in the set $\priv_\H(x, I) \coloneqq \{E\in \E(\H) \mid E\cap I=\{x\}\}$.
If $\H$ is clear from the context, we drop it from this notation.
We say that $I$ is an \emph{irredundant set} of $\H$ if every element $x$ of $I$ has a private edge, and that it is a \emph{redundant set} otherwise.

Observe that, given $\H$, one can check in polynomial time whether a set $I$ of vertices is redundant or irredundant by computing $\priv(x, I)$ for each $x \in I$.
We denote by $\maxirr(\H)$ the family of inclusion-wise maximal irredundant sets of $\H$, and by $\minred(\H)$ its family of inclusion-wise minimal redundant sets.
\iflongelse{%

    The problems of generating all maximal irredundant sets (resp.~all minimal redundant sets) thus read as follows:
    
    \begin{enumproblem}
    \problemtitle{Maximal Irredundant Sets Enumeration in Hypergraphs (\hypmirrenum{})}
    \probleminput{a hypergraph $\H$.}
    \problemquestion{the family $\maxirr(\H)$.}
    \end{enumproblem}
    
    \begin{enumproblem}
    \problemtitle{Minimal Redundant Sets Enumeration in Hypergraphs (\hypmredenum{})}
    \probleminput{a hypergraph $\H$.}
    \problemquestion{the family $\minred(\H)$.}
    \end{enumproblem}%
}{
    The two problems \hypmirrenum{} and \hypmredenum{} are respectively defined as producing, given $\H$, the sets $\maxirr(\H)$ and $\minred(\H)$.
}

We recall that these problems were shown intractable even on hypergraphs of dimension at most three \cite{boros2016wepa,boros2024generating} and refer to the introduction for the state of the art on these problems.
% These problems have been studied by Boros and Makino~\cite{boros2024generating} in terms of transposed hypergraphs, that is by considering sets of edges having private vertices instead of sets of vertices having private edges. 
% We translate their results in our language.
% First, they show that for hypergraphs of dimension 3, both problems are already intractable.
% For dimension at most 2, they prove that \graphmredenum{} can be solved in (input-)polynomial time, while \mirrenum{} is equivalent to the problem of enumerating all minimal dominating sets of a graph, hence to hypergraph dualization. 
% %, an open problem for which the best-known algorithm runs in output-quasipolynomial time \cite{fredman1996complexity}. % let's move this to the introduction ?
% Finally, for hypergraphs with bounded degree, they prove that \graphmredenum{} can be solved in (input-)polynomial time and \graphmirrenum{} in incremental-polynomial time.
% Let us also mention that a polynomial-delay algorithm has been given for \graphmirrenum{} on hypergraphs of bounded degeneracy in \cite{conte2019maximal}.

In this paper, we are interested in the analogues of redundancy and irredundancy in graphs, that we recall now.
Let $G$ be a graph, $I \subseteq V(G)$, and $x\in I$.
We call \emph{private neighbor} of $x$ the elements of the set $\priv_G(x,I) \coloneqq \{y\in N[x] \mid N[y] \cap I = \{x\}\}$.
Note that a vertex may be self-private.
Analogously, as for hypergraphs, a subset $I$ of vertices is called \emph{irredundant} if $\priv_G(x,I)\neq \emptyset$ for all $x\in I$, and \emph{redundant} otherwise.
We denote by $\maxirr(G)$ and $\minred(G)$ the set of maximal irredundant and minimal redundant sets of $G$.
To avoid any confusion between graphs and hypergraphs, we will only refer to hypergraphs using calligraphic letters.
\iflongelse{%

    The enumeration problems we are interested in are defined as follows.
    
    \begin{enumproblem}
    \problemtitle{Maximal Irredundant Sets Enumeration on Graphs (\graphmirrenum{})}
    \probleminput{a graph $G$.}
    \problemquestion{the family $\maxirr(G)$.}
    \end{enumproblem}
    \begin{enumproblem}
    \problemtitle{Minimal Redundant Sets Enumeration on Graphs (\graphmredenum{})}
    \probleminput{a graph $G$.}
    \problemquestion{the family $\minred(G)$.}
    \end{enumproblem}
}{
    The two problems \graphmirrenum{} and \graphmredenum{} are respectively defined as producing, given $G$, the sets $\maxirr(G)$ and $\minred(G)$.
}

As stated in the introduction, it is easily seen that $\maxirr(G)$ and $\minred(G)$ coincide with the maximal irredundant and minimal redundant sets of $\N(G)$.
They should, however, not be confused with the maximal irredundant and minimal redundant of $G$ seen as a hypergraph of dimension 2.

\section{Maximal irredundant sets}
\label{section:MIRR}

\iflongelse{
In this section we characterize the complexity of \graphmirrenum{} in split and co-bipartite graphs, as well as in strongly orderable graphs, which generalize chordal bipartite graphs.%
}{
In this section we characterize the complexity of \graphmirrenum{} in co-bipartite graphs and strongly orderable graphs, the latter being a generalization of chordal bipartite graphs.
We refer the reader to the appendix for a proof of the case of split graphs originally claimed by Uno in~\cite{bodlaender2015open}.
}

\subsection{Strongly orderable graphs}
\label{subsection:MIRR_strongly_orderable}

% - Fixed confusions around $i\in [n]$, $i\in [n-1]$, $1<i\leq n$\\
% - Defined $\H_0$ as the base case\\
% - No more notion of $G_i$ to avoid confusion with $\H_i$\\
% - Moved stars so that $I^{\star}$ is for the parent only\\
% - rewrote the parent definition\\
% - added a rough description of the strategy before the tree definition\\
% - changed the notion of extension to remove minimality\\
% - changed $R(I, i+1)$ to $R(I, i)$ to match children definition\\
% - explicited the link between $\H$ and $G$ in the proofs : (traced) edges are actually pieces of neighborhoods of $G$

We give an algorithm that generates all maximal irredundant sets of strongly orderable graphs with polynomial delay and space.
Our algorithm is based on the \emph{ordered generation} framework (also known as the \emph{sequential method}) that has been successfully applied to various instances admitting elimination orders~\cite{eiter2003new,bartier2024hypergraph,castelo2025enumerating}.
Our algorithm, however, differs from the strategy used in the aforementioned works by considering traces of the hypergraph (instead of its induced subhypergraphs) in the decomposition, following what has been done in \cite{conte2019maximal} for irredundant sets.
% The key insight is the structural properties ensured by the elimination ordering, which define this class of graphs.

Let \(G\) be a graph.
In the remaining, let us fix an arbitrary ordering \(v_{1}, \dots, v_{n}\) of the vertex set of \(G\), and let \(V_{i} \coloneqq \{v_{1}, \dots, v_{i}\}\) denote the set of its \(i\) first vertices.

Let \(\H \coloneqq \N(G)\) be the closed-neighborhood hypergraph of $G$, and \(\H_{i} \coloneqq \H_{V_{i}}\) be a shorthand for the trace of $\H$ on $V_i$. 
For each \(I \subseteq V_{i}\) and \(x \in I\), by \(\priv_{i}(x, I)\) we mean the set of private edges of \(x\) with respect to \(I\) in \(\H_{i}\), and by \(\inc_{i}(x)\) we mean the collection of hyperedges of \(\H_{i}\) that contain \(x\). 
Note that \(\maxirr(\H_{1}) = \{\{v_{1}\}\}\). 
This is due to \(N[v_{1}]\) being a hyperedge in \(\H\), hence the singleton \(\{v_{1}\}\) is the only hyperedge of \(\H_{1}\).

The goal of the algorithm will be to construct $\maxirr(\H_i)$ for $i$ ranging from 1 to $n$, from which we only output $\maxirr(\H_n)=\maxirr(\H)$.
We will call \emph{partial solutions} the maximal irredundant sets of $\H_i$ for $i<n$, and \emph{solutions} those of $\H_n$.
However, and to guarantee polynomial delay between consecutive outputs, we will not construct the sets $\maxirr(\H_i)$ one after the other.
Rather, we will define a tree over the set of solutions and partial solutions that will be traversed by our algorithm.

To define this tree, we start with an easy observation whose proof is omitted; see also~\cite{conte2019maximal}.
The intuition is that whenever a vertex $v_j$ has a private edge in $\H_i$ for some $i>j$, then it keeps that private edge in any subhypergraph $\H_k$, where $j\leq k\leq i$.

\begin{proposition}\label{prop:irr-without-viplusone}
    Let $i\in \{2, \dots, n\}$ and \(I \in \maxirr(\H_i)\).
    Then \(I \setminus \{v_i\}\) is an irredundant set of \(\H_{i-1}\).
\end{proposition}

% To construct a solution tree, i.e., a tree where each vertex represents a maximal irredundant set of \(\H_{i}\) for some \(i \in [n]\), it is useful to define a notion of \emph{parent} of the vertices.
Let $i\in \{2, \dots, n\}$ and \(I \in \maxirr(\H_{i})\).
The \emph{parent of $I$ with respect to $i$}, denoted \(\parent(I, i)\), is defined as the set obtained by the following procedure.
First, if $v_i$ belongs to \(I\), we remove it.
Then, while there exists a vertex \(x \in V_{i-1} \setminus I\) such that \(I \cup \{x\}\) is irredundant in $\H_{i-1}$, we add the smallest such vertex \(x\) into \(I\).
Note that the obtained set is uniquely defined.
By Proposition~\ref{prop:irr-without-viplusone}, it is an irredundant set of $\H_{i-1}$ after the first step, and it is extended into a maximal irredundant set of $\H_{i-1}$ at the end of the procedure.
By construction we thus get 
\[\parent(I, i) \in \maxirr(\H_{i-1})\]
and \(\parent(I, i)=I\) whenever $v_i\not\in I$.

Conversely, given \(i \in [n-1]\) and $I^{\star}\in \maxirr(\H_i)$, we define the \emph{children of $I^{\star}$ with respect to $i$} as \(\children(I^{\star}, i) \coloneqq \{I \in \maxirr(\H_{i+1}) \mid \parent(I, i + 1) = I^{\star}\}\).

The following shows that any partial solution has at least one child, and that it has precisely one child if it can be extended into an irredundant set of $\H_{i+1}$ by adding $v_{i+1}$.

\begin{lemma}
    \label{lemma:MIRR_parent_function}
	Let \(i \in [n-1]\) and \(I^{\star} \in \maxirr(\H_{i})\).
	Then, either: 
	\begin{itemize}
		\item \(I^{\star} \in \maxirr(\H_{i+1})\), in which case \(I^{\star}\in \children(I^{\star}, i)\); or
		\item \(I^{\star} \cup \{v_{i+1}\} \in \maxirr(\H_{i+1})\), in which case \(\children(I^{\star}, i) = \{I^{\star}\cup \{v_{i+1}\}\}\).
	\end{itemize}
	
	\begin{proof}
		Let \(I^{\star} \in \maxirr(\H_{i})\).
        By definition, $I^{\star}$ does not contain $v_{i+1}$, and no vertex in $V_i$ can be added to $I^{\star}$ while maintaining the irredundancy of $I^{\star}$ in $\H_i$.
        Hence the $\parent$ procedure, when performed on $(I^{\star},i+1)$ in case \(I^{\star} \in \maxirr(\H_{i+1})\), does not modify the set and gives \(\parent(I^{\star}, i+1) = I^{\star}\), thus \(I^{\star}\in \children(I^{\star}, i)\).
        This proves the first assertion of the claim.
		
        Let us now assume that the first case does not arise, i.e., \(I^{\star} \not\in \maxirr(\H_{i+1})\).
		Suppose, for the sake of contradiction, that \(I^{\star} \cup \{v_{i+1}\}\) is redundant.
		Since \(I^{\star} \in \maxirr(\H_{i})\) every set \(I^{\star} \cup \{x\}\) with \(x \in V_{i} \setminus I^{\star}\) is redundant in \(\H_{i}\).
        The same is true in $\H_{i+1}$ since a private edge of such a $x$ in \(\H_{i+1}\) would yield one (possibly reduced to its intersection with $V_i$) in $\H_i$.
		However, this implies that \(I^{\star}\) is inclusion-wise maximal with respect to being irredundant in \(\H_{i+1}\), contradicting our assumption.
		So \(I^{\star} \cup \{v_{i+1}\}\) is irredundant.
        As $\H_i$ and $\H_{i+1}$ only differ by the trace of the edges incident to $v_{i+1}$, the set $I^{\star}\cup \{v_{i+1}\}$ must be maximal with that property.
        % We derive that \(\parent(I^{\star} \cup \{v_{i+1}\}, i+1) = I^{\star}\), and hence that \svtodo{not $I^\star = \children(I^\start, i)$} $\{I^{\star} \cup \{v_{i+1}\}\}= \children(I^{\star}, i)$ in that case.}
        By definition of the parent relation, we thus obtain $\{I^\star \cup \{v_{i+1}\}\} \subseteq \children(I^\star, i)$.
        The other inclusion follows from the fact that any $I' \in \children(I^\star, i)$ satisfies $I' \setminus \{v_{i+1}\} \subseteq I^\star$, hence $I' \subseteq I^\star \cup \{v_{i+1}\}$.
	\end{proof}
\end{lemma}

Note that the $\parent$ relation defines a tree $\T$ on the set of nodes $\{(I,i) \mid I\in \maxirr(\H_i),\allowbreak\ 1 \leq i\leq n\}$, whose root is $(\{v_{1}\}, 1)$, and where there is an edge between two nodes $(I^{\star},i)$ and $(I,i+1)$ if $I\in\children(I^{\star},i)$.
Moreover, by Lemma~\ref{lemma:MIRR_parent_function}, every node $(I,i)$ corresponding to a partial solution $I$ has a child. Hence the set of leaves of this tree is precisely the family $\{(I,n) \mid I\in \maxirr(\H)\}$ we wish to enumerate.

In the following, we show that it is sufficient to be able to list the children with polynomial delay and space to derive a polynomial delay and space algorithm for \graphmirrenum{}.

\begin{proposition}
    \label{proposition:MIRR_generation_algorithm}
    Let \(f, s \colon \mathbb{N} \rightarrow \mathbb{N}\) be two functions.
	% Fix an ordering \(v_{1}, \dots, v_{n}\) of \(V(G)\).
    Suppose that there is an algorithm that, for any \(i \in [n-1]\) and \(I^{\star} \in \maxirr(\H_{i})\), generates $\children(I^{\star}, i)$ with \(f(n)\) delay and \(s(n)\) space. Then there is an algorithm that generates \(\maxirr(\H)\) with \(O(n \cdot f(n))\) delay and \(O(n \cdot s(n))\) space.
	
	\begin{proof}
		% Our goal is to generate all maximal irredundant sets of \(\H\) with polynomial delay and space. \odtodo{proof omitted for now.}
		% To achieve this, we construct a search tree whose vertices represent partial solutions, and then perform a depth-first search from the root.
	       %
        Let \(\algoA\) be the children-generation algorithm as given by assumption.
        We describe an algorithm \(\algoB\) which generates all maximal irredundant sets with the aforementioned worst-case guarantees as follows.
		% For the base case we set \(\parent(I, 1) = \emptyset\) for each \(I \in \maxirr(\H_{1})\).
		% The parent relation then defines a directed graph \(G\) with node set \(\{(I, i) \mid i \in [n], I \in \maxirr(\H_{i})\} \cup \{(\emptyset, 0)\}\) where there is an arc from \((I, i)\) to \((I^{\star}, j)\) if \(I = \parent(I^{\star}, j)\) and \(j = i + 1\).
  %       Since \(\parent\) is uniquely defined for each node except \((\emptyset, 0)\), the directed graph \(G\) must be a directed tree with \((\emptyset, 0)\) as the root.

        The algorithm performs a DFS on $\T$ starting from the root $(\{v_{1}\}, 1)$ and outputs each leaf as it is visited.
        Its correctness follows by the above discussion.
        Since every leaf is at depth \(n\), and for every \(i \in [n-1]\) and \(I \in \maxirr(\H_{i})\) the children of \((I, i)\) can be computed with \(f(n)\) delay, a first maximal irredundant set of \(\H\) is obtained in \(O(n \cdot f(n))\) time.
        As leaves are at a distance at most \(2n\) in $\T$, the time spent by the DFS between consecutive outputs is bounded by \(O(n \cdot f(n))\) as well. 
        Finally, after the last output, the algorithm concludes that no other solution exists within the same time after having backtracked to the root and concluding that it has explored all its children.
        
        Note that the algorithm requires \(O(n \cdot s(n))\) space for maintaining the recursion stack for backtracking.
        This concludes the proof.
	\end{proof}
\end{proposition}

We now focus on generating $\children(I^{\star},i)$ for fixed \(i \in [n-1]\) and \(I^{\star} \in \maxirr(\H_{i})\) such that \(I^{\star}\cup \{v_{i+1}\}\) is redundant.
Recall by Lemma~\ref{lemma:MIRR_parent_function} that $I^{\star}$ is one of these children, and we may thus focus on those containing $v_{i+1}$, which we shall call \emph{non-trivial children}.

Our strategy is to find all sets of the form \(Z \coloneqq (I^{\star} \setminus X) \cup \{v_{i+1}\}\), where \(X \subseteq I^{\star}\), that are maximal irredundant sets of \(\H_{i+1}\).
We call each such set \(Z\) an \emph{extension of $I^{\star}$ to $i+1$}.
Note that these extensions do not define supersets of $I^{\star}$, as they are obtained by adding $v_{i+1}$ and removing elements of $I^{\star}$.
As another remark, note that possibly not all extensions of $I^{\star}$ to $i+1$ are children of $I^{\star}$ with respect to $i$. 
However, all non-trivial children $I$ of $I^{\star}$ are extensions of $I^{\star}$ to $i+1$:
by definition, if $I^\star = \parent(I)$ then $I^{\star}$ is of the form $(I\setminus \{v_{i+1}\})\cup X$ for some $X\subseteq V_{i}$, and so $I=(I^{\star} \setminus X) \cup \{v_{i+1}\}$ with $X\subseteq I^\star$.
% a vertex can be added during the algorithm $\parent$ only if, after removing $v_{i+1}$ from $I$, this vertex does not dominate all the private neighbors of a vertex in \(I \setminus \{v_{i+1}\}\).

% \textcolor{red}{We aim to find the subsets \(X \subseteq R(I^{\star}, i)\) such that \((I^{\star} \cup \{v_{i+1}\}) \setminus X\) is a maximal irredundant set of \(\H_{i+1}\). 
% In other words, we want to remove vertices of $R(I^{\star},i)$ so that the remaining ones gain private neighbors.
% }\ectodo{This is incorrect. See algorithm in Lemma 3.6.}\ectodo{I am thinking of a good explanation to replace it.}

% Note that the sets $X$ as described above are the vertices to remove vertices of \(I^{\star}\) such that those in \(R(I^{\star}, i) \setminus X\) gain private neighbors.

% In the first case, it follows that such a private edge must contain $v_{i+1}$, and thus this vertex is at a distance of at most 2 from $v_{i+1}$ in $G[V_{i+1}]$.
% Otherwise, there exists a vertex in \(I \setminus \{v_{i+1}\}\) at a distance at most \(2\) from \(v_{i+1}\) that gets a new private edge, and hence adding \(y\) does not dominate all its private neighbors, in which case the vertex must be at distance at most \(4\) from \(v_{i+1}\) in \(G[V_{i+1}]\).
% We will argue that we can afford to compute all extensions and sort out children while preserving polynomial delay and space, with an extra assumption on the structure of $G$.

Let us now assume that \(G\) is strongly orderable and that \(v_{1}, \dots, v_{n}\) is a quasi-simple elimination ordering of its vertices.
In the following, by $N_i(x)$ we mean the neighborhood of $x$ intersected with $V_i$.

To compute all extensions efficiently, we will rely on a characterization of their intersection with the set of vertices that become redundant when adding $v_{i+1}$ to $I^\star$.
More formally, let \(R(I^{\star}, i) \coloneqq \{x \in I^{\star} \mid \priv_{i+1}(x, I^{\star}) \subseteq \inc_{i+1}(v_{i+1})\}\) denote such a set.
Note that the vertices in \(R(I^{\star}, i)\) lie at a distance of at most 2 from $v_{i+1}$ in $G[V_{i+1}]$.
An upper bound on the number of elements in \(R(I^{\star}, i)\) found in an extension of $I^{\star}$ is given in the following lemma.

\begin{lemma}\label{lemma:MIRR:charac_strongly_orderable}
    Let Z be an extension of $I^\star$ to $i+1$.
    Then:
    \begin{itemize}
        \item \(|R(I^{\star}, i) \cap N_{i+1}(v_{i+1}) \cap Z| \leq 1\); and 
        \item \(|R(I^{\star}, i) \cap N^{2}_{i+1}(v_{i+1})| \leq 1\).
    \end{itemize}

    \begin{proof}
        By definition of $R(I^{\star},i)$, the private edges of an element $x\in R(I^{\star},i)$ with respect to $I^{\star}$ in $\H_{i+1}$ all contain $v_{i+1}$.
        By the definition of $\H$, each such private edge is thus of the form $N_{i+1}[u]$ for some $u\in N_{i+1}[v_{i+1}]$, with possibly $u=x$, i.e., $x$ is self-private and lies in the neighborhood of $v_{i+1}$.
        Recall that by definition, $Z$ is a maximal irredundant set of $\H_{i+1}$, and so each of its elements has a private edge in $\H_{i+1}$.
        
        Assume, for the sake of contradiction, that there are at least two distinct vertices $x_{1}$ and $ x_{2}$ in $R(I^{\star}, i) \cap N_{i+1}(v_{i+1}) \cap Z$.
        Since \(v_{i+1}\) is quasi-simple, $x_{1}$ and $x_{2}$ are pairwise comparable in $G[V_{i+1}]$. 
        So the only way for $x_{1}$ and $x_{2}$ to both have a private neighbor as described above is for one to be self-private with respect to $I^{\star}$ in $\H_{i+1}$.
        However, by definition, $v_{i+1}$ belongs to $Z$.
        So $x_1$ and $x_2$ cannot be self-private with respect to $Z$.
        We conclude that one of the two is redundant in $Z$, which contradicts our assumption and proves the first item.

        Additionally, from the fact that \(v_{i+1}\) is quasi-simple, we also derive that for every pair of distinct vertices \(y_{1}, y_{2} \in N^{2}_{i+1}(v_{i+1})\), either \(N_{i+1}(y_{1}) \cap N_{i+1}(v_{i+1}) \subseteq N_{i+1}(y_{2}) \cap N_{i+1}(v_{i+1})\) or the opposite holds.
        Consequently, a pair of vertices in \(N_{i+1}^{2}(v_{i+1})\) can neither simultaneously have private edges as described above.
        Hence, they cannot both have a private edge in $\H_{i+1}$, proving the second item of the claim.
    \end{proof}
\end{lemma}

Let us point out that $Z$ cannot be removed from the above statement. 
In fact, the cardinality of $R(I^{\star}, i)\cap N_{i+1}(v_{i+1})$ is unbounded in general.
This can be seen by considering the situation where $N_{i+1}(v_{i+1})$ is an independent set whose elements are self-private with respect to $I^{\star}$.

Additionally, let us point that the set \(X \subseteq I^{\star}\) such that \(Z =(I^{\star} \setminus X) \cup \{v_{i+1}\}\) is an extension may contain vertices that are not in $R(I^{\star}, i)$.
Indeed, removing vertices from $I^\star\setminus R(I^{\star}, i)$ may provide private edges to the remaining vertices in $R(I^{\star}, i)$.

However, from Lemma~\ref{lemma:MIRR:charac_strongly_orderable} we know that at most two vertices from $R(I^{\star}, i)$ may be kept in an extension, from which we derive the following.

\begin{lemma}
\label{lem:MIRR:bounded_number_of_candidate}
	% Let \(i \in [n-1]\) and \(I \in \maxirr(\H_{i})\).
    The number of extensions of $I^{\star}$ to $i+1$ is bounded by $O(n^{3})$, and these extensions can be computed in polynomial time  in $n$.
	% hence the size of $\children(I, i+1)$, Moreover, these sets can be computed in polynomial time in $n$.
\end{lemma}

\begin{proof}
    By \Cref{lemma:MIRR:charac_strongly_orderable} each extension $Z$ intersects $R(I^\star,i)$ on at most two vertices: one possibly being the unique vertex $x$ in $R(I^\star,i)\cap N_{i+1}^2(v_{i+1})$ if it exists, and the other vertex $y$ being among those in $R(I^\star,i)\cap N_{i+1}(v_{i+1})$, if it exists.
    These vertices $x$ and $y$ have at least one private edge each.
    Note that all other vertices in $Z$ are in $I^\star\setminus R(I^\star,i)$, hence have private edges with respect to $I^\star$.
    So the only reason for vertices in $I^\star\setminus R(I^\star,i)$ not to be included in $Z$ is because they participate to intersect all private edges of $x$ or $y$; otherwise they could be added to $Z$, contradicting its maximality.
    We conclude that the set of all extensions $Z$ can be found by the following procedure.
    First we decide whether to select the vertex $x\in R(I^\star,i)\cap N_{i+1}^2(v_{i+1})$, if it exists, and then select at most one, possibly none, vertex $y \in R(I^\star,i)\cap N_{i+1}(v_{i+1})$.
    For each of the at most two selected vertices \(x\) and \(y\), we choose one of their incident edges $E_x \in \inc_{i+1}(x)$ and $E_y \in \inc_{i+1}(y)$ to become their respective private edges.
    The first two steps amount to removing $R(I^\star,i)$ from $I^\star$ except for the possibly selected vertices \(x\) and \(y\).
    The second step is to remove $(E_x \setminus \{x\}) \cup (E_y \setminus \{y\})$ from $I^\star$.
    For each obtained set we check whether it is a maximal irredundant set and output it if this is the case.
    Note that the total number of different selections is bounded by $O(n^3)$ and that this procedures indeed runs in polynomial time in the number of vertices.
\end{proof}

\iflongelse{We are ready to show}{We get} that generating the children of a partial solution can be done with polynomial delay and space, by enumerating all extensions.

\begin{lemma}
    \label{lemma:MIRR_child_generation_algorithm}
    There is an algorithm that generates \(\children(I^{\star}, i)\) in polynomial time.
	%There is an algorithm that generates \(\children(I^{\star}, i)\) with \(O(n^{5})\) delay and \(O(n^{2})\) space.

	\begin{proof}
        By Lemma~\ref{lemma:MIRR_parent_function}, two cases arise.
        Either \(I^{\star}\cup \{v_{i+1}\}\in \maxirr(\H_{i+1})\), in which case we output it and are done with the enumeration, or \(I^{\star} \in \maxirr(\H_{i+1})\), in which case we output it and are left with the generation of the non-trivial children of $I^\star$.
        %Either \(I^{\star} \in \maxirr(\H_{i+1})\), in which case we output it and are left with the generation of non-trivial children, or \(I^{\star}\cup \{v_{i+1}\}\in \maxirr(\H_{i+1})\), in which case we output it and are done with the enumeration.
        To enumerate them, we start enumerating all extensions in polynomial time using \Cref{lem:MIRR:bounded_number_of_candidate}.
        For every extension, we check whether it is a child of $I^{\star}$ by running the $\parent$ procedure and output it if it is positive. 
        % Since \(|Z| \leq 2\) there are \(O(n^{2})\) ways of choosing \(Z\) and \(O(n^{2})\) ways of selecting a pair of private hyperedges for the elements of \(Z\).
        % Moreover, verifying whether $I$ is a maximal irredundant set and computing the parent relation can be done in polynomial time.
        As every step takes polynomial time and since $I^\star$ has a polynomial number of children by \Cref{lem:MIRR:bounded_number_of_candidate}, we obtain the desired result.
	\end{proof}
\end{lemma}

\iflongelse{
We conclude to \Cref{thm:MIRR:strongly_orderable} that we restate here as a corollary of~\Cref{proposition:MIRR_generation_algorithm} and \Cref{lemma:MIRR_child_generation_algorithm}. 

\thmMIRRstronglyorderable*
}{
We conclude to \Cref{thm:MIRR:strongly_orderable} as a corollary of~\Cref{proposition:MIRR_generation_algorithm} and \Cref{lemma:MIRR_child_generation_algorithm}. 
}

\subsection{Co-bipartite graphs}

\iflongelse{
In this section, we show that \graphmirrenum{} on co-bipartite graphs is as hard as \graphmirrenum{} on arbitrary graphs.
Our reduction considers the co-bipartite incidence graph $C(\N(G))$ of the closed-neighborhood hypergraph $\N(G)$ of a given graph $G$; see Figure~\ref{fig:cobip-mirr} for an illustration.
}{
To show that \graphmirrenum{} on co-bipartite graphs is as hard as \graphmirrenum{} on arbitrary graphs, we consider the co-bipartite incidence graph $C(\N(G))$ of the closed-neighborhood hypergraph $\N(G)$ of a given graph $G$; see Figure~\ref{fig:cobip-mirr} for an illustration.
}
Using the fact that $C(\N(G))$ is symmetric, we prove that listing the maximal irredundant sets of $C(\N(G))$ amounts to list each maximal irredundant set of $G$ twice, plus a polynomial number of other solutions of size 2 spread across the parts of $C(\N(G))$.
\iflongelse{\thmMIRRcobip*}{We refer to the appendix for the details. This proves \Cref{thm:MIRR:cobip}.}

% \begin{theorem} \label{thm:MIRR:cobip}
% There is an output-polynomial time for \graphmirrenum{} on co-bipartite graphs if and only if there is one for \graphmirrenum{} on arbitrary graphs.
% \end{theorem}

\begin{figure}
    \centering
    \includegraphics[width=0.9\linewidth]{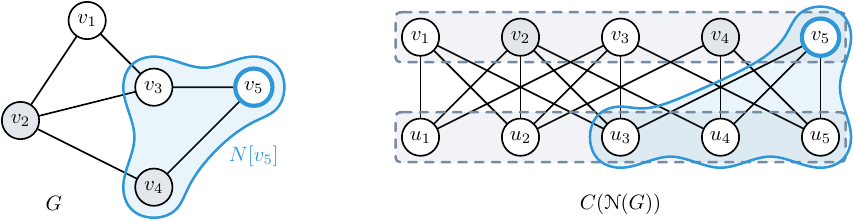}
    \caption{The reduction of Theorem~\ref{thm:MIRR:cobip}. On the left a graph $G$ where shaded vertices indicate a maximal irredundant set.
    On the right the incidence co-bipartite graph $C(\N(G))$ of $\N(G)$. Grey dotted zones indicate the clique bipartition and the shaded vertices form the maximal irredundant set $C(\N(G))$ induced by the maximal irredundant set in pictured in $G$. 
    The vertices boxed in blue illustrate how $N[v_5]$ (in $G$) is encoded in $C(\N(G))$.}
    \label{fig:cobip-mirr}
\end{figure}

\begin{proof}
The if part follows from the fact that an algorithm for \graphmirrenum{} for arbitrary graphs can be used in particular on a co-bipartite graph.
Let us thus show the only if part.
Let $G$ be a non-empty graph with vertex set $V = \{v_1, \dots, v_n\}$.
Let us consider the incidence co-bipartite graph $C \coloneqq C(\N(G))$ of the closed-neighborhood hypergraph $\N(G)$ of $G$.
The two parts of $C$ are $V$ and $U = \{u_1, \dots, u_n\}$ where $u_i$ represents $N[v_i]$.
Note that $C$ is symmetric: we have $u_iv_j\in E(C)$ if and only if $u_jv_i\in E(C)$, for all $i,j\in \{1,\dots,n\}$.
Moreover, the incidences between $U$ and $V$ coincide with those in $G$: we have $u_iv_j\in E(C)$ if and only if $v_iv_j\in E(G)$, for all $i\neq j\in \{1,\dots,n\}$.
For any $I \subseteq V$, let us denote by $U(I) \coloneqq \{u_i \mid v_i \in I\}$ the ``copy'' of $I$ in $U$.

We shall first prove that 
\begin{align*}
\maxirr(C) & = \mathcal{X} \cup \mathcal{I}_1 \cup \mathcal{I}_2 \quad \text{ with }\\
 \mathcal{X}\,& \coloneqq \{X \mid X\in \maxirr(C),\ |X|\leq 2 \}\\
 \mathcal{I}_1 & \coloneqq \{I \mid I \in \maxirr(G),\ |I|\geq 3\} \\ 
 \mathcal{I}_2 & \coloneqq \{U(I) \mid I \in \maxirr(G),\ |I|\geq 3\}
\end{align*}
where the sets $\mathcal{X}$, $\mathcal{I}_1$, and $\mathcal{I}_2$ are purposely refined by cardinality in order to avoid unnecessary technicalities in the proof.

Let us first prove the inclusion $\maxirr(C) \subseteq \mathcal{X} \cup \mathcal{I}_1 \cup \mathcal{I}_2$.
Let $I\in \maxirr(C)$.
Let us further assume that $|I|\geq 3$ as the inclusion trivially holds for smaller $I$. 
Note that $I$ is either a subset of $V$, or a subset of $U$, as otherwise it contains a proper dominating subset, and thus would be redundant.

If $I\subseteq V$, since $|I|\geq 3$, no element of $I$ has a private neighbor in $V$.
Hence $I$ is maximal with the property that its elements have private neighbors in $U$.
Now, since $v_iu_j\in E(C)$ if and only if $v_iv_j\in E(G)$, $I$ is actually maximal with the property that its elements have private neighbors in $G$.
Thus $I\in \mathcal{I}_1$.

The case $I\subseteq U$ leads to the conclusion that $I\in \mathcal{I}_2$ by symmetric arguments.

Let us now prove the other inclusion.
Let $I\in \mathcal{I}_1$.
As $v_iu_j\in E(C)$ if and only if $i =j$ or $v_iv_j\in E(G)$, every element of $I$ has a private neighbor in $U$.
%As $v_iv_j\in E(G)$ if and only if $v_iu_j\in E(C)$, every element of $I$ has a private neighbor in $U$.
So $I$ is irredundant in $C$, and by the maximality of $I$ in $G$, adding any vertex $v\in V\setminus I$ to $I$ makes an element of $I \cup \{v\}$ redundant with respect to $U$.
Adding a vertex $u$ of $U$ to $I$ would make all vertices of $I$ redundant with respect to $U$, since $U$ is a clique.
%The same is true when adding a vertex of $U$ to $I$, since $U$ is a clique, and that vertex would see these private %neighbors.
Now, since $|I|\geq 3$, no element of $I$ has a private in $V$.
We conclude that $I$ is a maximal irredundant set of $C$.

The case of $I\in \mathcal{I}_2$ leads to $I\in \maxirr(C)$ by symmetric arguments.

We are now ready to conclude the proof.
Assume there is an algorithm $\algoA$ solving \graphmirrenum{} on any co-bipartite graph $C$ in output-polynomial time.
We give an output-polynomial time algorithm solving \graphmirrenum{} for any graph $G$.
First the algorithm outputs all members of $\maxirr(G)$ of size at most 2 by ranging over subsets of size 2 and testing if they define a maximal irredundant set of $G$.
Then, the algorithm builds $C$, runs $\algoA$ on $C$, and outputs only those members of $\maxirr(G)$ of size at least 3. 
These steps are done in a time which is polynomial in the sizes of $C$ and $\maxirr(C)$.
Due to the above equality, each solution will be correctly output.
Again using this equality, we have $\vert \maxirr(C) \vert \leq 2 \times \vert \maxirr(G) \vert + \vert V(G) \vert^2$.
Therefore, the whole algorithm runs in polynomial time in the sizes of $G$ and $\maxirr(G)$ as expected.
This concludes the proof.
\end{proof}

\iflongelse{%
We note that the reduction in the proof of Theorem~\ref{thm:MIRR:cobip} actually preserves incremental polynomial time, since the procedure may list linearly many solutions that correspond to already obtained solutions before outputting a new solution.}{}

\iflongelse{
\subsection{Split graphs}

Finally, for completeness, we include a proof of Theorem~\ref{thm:MIRR:split} which was initially claimed by Uno in \cite{bodlaender2015open}.

\begin{lemma} \label{lem:MIRR-split-MDS}
The maximal irredundant sets of a split graph $G$ are precisely its minimal dominating sets.
\end{lemma}

\begin{proof}
Recall that in any graph, every minimal dominating set is a maximal irredundant set; see Section~\ref{section:preliminaries}.
We show that the converse is true in split graphs.
Let $G$ be a split graph of clique-independent set partition $(C,S)$, and $I\in \maxirr(G)$.
Let us assume without loss of generality that $S$ is maximized in this partition, i.e., every element of $C$ has a neighbor in $S$.
Note that for any $y\in I\cap S$, we have that $I\cap N(y)=\emptyset$, as otherwise $y$ is redundant.
Consequently, every element of $I\cap S$ is self-private, and every element of $I\cap C$ has a private neighbor in $S$.

Suppose toward a contradiction that $I$ is not a dominating set.
If a vertex $y$ of the independent set is not dominated, then $y$ is self-private with respect to $I\cup \{y\}$.
Moreover, by the above discussion, every other vertex in $I$ either is self-private if it belongs to $S$, or has a private neighbor in $S$. 
So $I\cup \{y\}$ is irredundant, contradicting the maximality of $I$.
On the other hand, if a vertex $x$ of the clique is not dominated, then it has a neighbor $y$ in the independent set that is not dominated.
Conducting the same argument as before, $I\cup \{y\}$ is irredundant, a contradiction.
\end{proof}

We conclude to Theorem~\ref{thm:MIRR:split} using a linear delay and space algorithm for minimal dominating sets enumeration in split graphs, as given in \cite{kante2014enumeration}.
}{}

\section{Minimal redundant sets}
\label{section:MRED}

In this section, we investigate the complexity of \graphmredenum{}.
We first propose a polynomial-delay algorithm for the class of $(C_3, C_5, C_6, C_8)$-free graphs.
Recall that this class contains chordal bipartite graphs, which are precisely $(C_3, C_{\geq 5})$-free graphs \cite{golumbic1978perfect}.
\iflongelse{Then, we prove that the problem is intractable for co-bipartite and split graphs.}{Then, we prove that the problem is intractable for co-bipartite, and similar arguments (deferred to the appendix) show the same for split graphs.}

\subsection{Graphs excluding small cycles}

In this section, we characterize the minimal redundant sets of graphs excluding small cycles.
Our characterization relies on the number of redundant vertices in a minimal redundant set and involves minimal red-blue dominating sets within a ball of radius~2. 
This yields a polynomial-delay algorithm for \graphmredenum{} on $(C_3, C_5, C_6, C_8)$-free graphs---which generalize chordal bipartite graphs---relying on existing results from the literature on red-blue domination, and an output quasi-polynomial-time algorithm on $(C_3, C_5, C_6)$-free graphs.

The remainder of the section is organized as follows.
We start by stating preliminary properties in general graphs in \Cref{sec:MRED:general-prop}, and then turn to graphs with no small cycles. 
We prove properties for these graphs in \Cref{sec:MRED:cycle-free-prop}, provide the aforementioned characterization of minimal redundant sets in \Cref{sec:MRED:cycle-free-characterization}, and describe the algorithm in \Cref{sec:MRED:cycle-free-algorithm}.

\subsubsection{Preliminary properties in graphs}\label{sec:MRED:general-prop}

If $R$ is a redundant set, we denote by \(\red(R) \coloneqq \{x \in R \mid \priv(x, R) = \emptyset\}\) the set of its redundant vertices.
The next lemma states that in a minimal redundant set $R$ with redundant vertex $x$, any vertex other than $x$ prevents $x$ from having a dedicated private neighbor.

% \jctodo{If I understand correctly the proof of the lemma, we prove that \(\priv(x,R\setminus\{y\}) \subseteq \priv(y,R\setminus\{x\})\). Don't we want to state it explicitly? Cannot it be useful later in the proof? In particular, if $x,y \in \red(R)$, it implies  \(\priv(x,R\setminus\{y\}) = \priv(y,R\setminus\{x\})\).
% } % \odtodo{I don't think we need it}
\begin{lemma}
	\label{lemma:MRED_existence_of_privates}
    Let \(G\) be a graph and \(R \in \minred(G)\).
    Then for any two distinct vertices \(x, y \in R\) such that \(x \in \red(R)\) there exists a vertex \(z \in \priv(x, R \setminus \{y\}) \cap \priv(y, R \setminus \{x\})\).

    \begin{proof}
        By the minimality of \(R\), the vertex \(y\) satisfies \(\priv(x, R \setminus \{y\}) \neq \emptyset\).
        Let \(z \in \priv(x, R \setminus \{y\})\).
        Because \(z\) is a private neighbor of \(x\) in \(R \setminus \{y\}\) and \(x\) is redundant in \(R\), it follows that \(z\) is adjacent to $y$, nonadjacent to \(R \setminus \{x, y\}\), and thus \(z \in \priv(y, R \setminus \{x\})\).
        This concludes the proof.
    \end{proof}
\end{lemma}

The minimality of the redundant sets also implies the following.

\begin{proposition}
    \label{prop:MRED:maximum_distance}
    Let \(G\) be a graph and \(R \in \minred(G)\) be a minimal redundant set.
    Then \(R \subseteq N^2[x]\) for every \(x \in \red(R)\). 
\end{proposition}

% Consequently, fixing a vertex to be part of a solution also implies that our decisions are limited to a small subgraph around the said vertex, and therefore we consider classes of graphs for which this subgraph has algorithmically interesting properties. \svtodo{not sure this paragraph helps a lot}

\subsubsection{Properties in graphs excluding small cycles}\label{sec:MRED:cycle-free-prop}

\begin{lemma}
    \label{lem:MRED:bipartite-neighborhood}
    Let \(G\) be a \((C_3, C_5)\)-free graph.
    Then for every vertex \(x \in V(G)\), the graph 
    induced by \(N^2[x]\) is bipartite.

    \begin{proof}
        Let \(x \in V(G)\).
        As $G$ is $C_3$-free, $N(x)$ is an independent set.
        Let us suppose toward a contradiction that there are two adjacent vertices $z_{1}, z_{2}$ in \(N^{2}(x)\).
        Let $y_1$ and $y_2$ be neighbors of $z_1$ and $z_2$ in $N(x)$, respectively.
        As $G$ is $C_3$-free, $y_1\neq y_2$.
        But then \(\{z_{1}, y_{1}, x, y_{2}, z_{2}\}\) induces a cycle of length \(5\), a contradiction.
        So \(N^{2}(x)\) induces an independent set and the statement follows.
    \end{proof}
\end{lemma}

In the following, we show that a minimal redundant set containing a redundant vertex $x$ and one of its neighbor $y$ %containing an edge $xy$ 
such that $N(y)\neq \{x\}$ may contain other elements in $N(y)$, or in $N(x)$, but not in both sets at the same time, while it should intersect exactly one if $y$ is redundant.
This is illustrated in Figure~\ref{fig:R-location-C3C5-free}.

\begin{lemma}
    \label{lemma:MRED_cases_two_redundant_vertices}
    Let \(G\) be a \((C_3, C_5)\)-free graph, \(R \in \minred(G)\) be a minimal redundant set, and \(x \in \red(R)\) be a redundant vertex.
    If \(y \in R\) is adjacent to \(x\) and \(N(y)\neq \{x\}\), then at most one of the following holds:
    \begin{itemize}
        \item \(R \cap (N(y) \setminus \{x\}) \neq \emptyset\); or
        \item \(R \cap (N(x) \setminus \{y\}) \neq \emptyset\).
    \end{itemize}
    Moreover, if \(y \in \red(R)\) then exactly one of the items is satisfied.

    \begin{proof}
        By contradiction.
        Assume both items hold.
        The second item implies that \(x\) does not become self-private after removing a vertex from \(R \cap N(x)\).
        From the first item, the vertex \(y\) does not become a private neighbor of \(x\) in \(R \setminus \{y\}\).
        Since \(G\) is $C_3$-free, the vertices \(x\) and \(y\) do not share a common neighbor, and thus \(x \in \red(R \setminus \{y\})\), a contradiction to the minimality of \(R\).

        It remains to show that at least one of the items holds when \(y \in \red(R)\).
        If \(y \in \red(R)\) the set \(R \setminus \{y\}\) minimally dominates \(N[y]\).
        Let $z \in N(y)\setminus \{x\}$ and assume that $z \notin R$, as otherwise the first item holds. Then there exists $w \in R\setminus \{y\}\cap N(z)$. 
        By~\Cref{prop:MRED:maximum_distance}, \(R \subseteq N^{2}[x]\). By~\Cref{lem:MRED:bipartite-neighborhood}, $w \in N(x)$ and since $w \neq y$, the second item holds.
        %So \(R \setminus \{y\}\) must contain at least one element in \(N(x)\setminus \{y\}\), or an element in $N(y)\setminus \{x\}$ as \(N(y)\neq \{x\}\).
        This concludes the proof.
    \end{proof}
\end{lemma}

\begin{figure}
    \centering
    \includegraphics[scale=0.70]{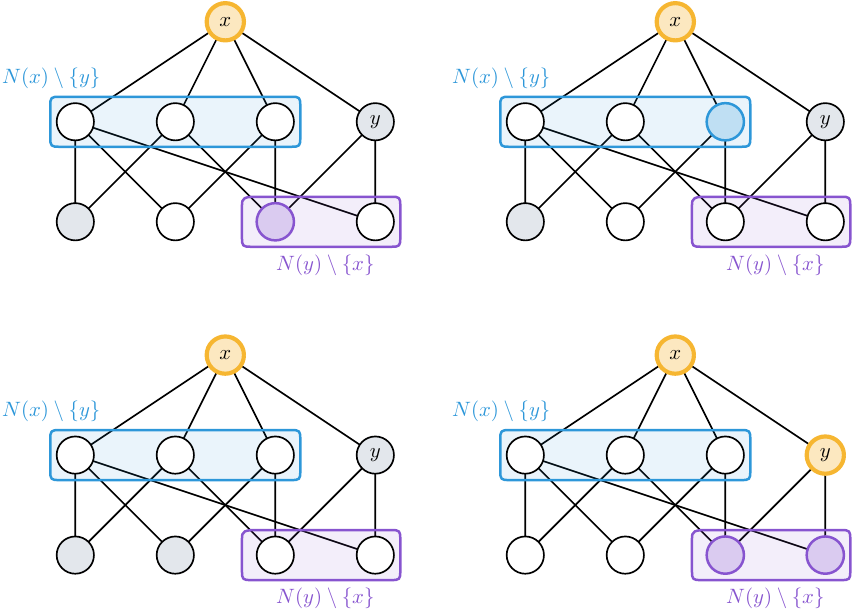}
    \caption{The situations of \Cref{lemma:MRED_cases_two_redundant_vertices}. In each case, a minimal redundant set is indicated by shaded vertices with (yellow) bold vertices being redundant.
    The top-left, top-right, and bottom-left cases illustrate that $R$ intersects at most one of $N(x) \setminus \{y\}$ or $N(y) \setminus \{x\}$. The bottom-right graph illustrates one of the two cases of \Cref{lemma:MRED_redundant_contained_closed_neighborhood_of_x}, where $y$ is also redundant: \(R = N[y]\). 
    %$R \setminus \{x, y\} \subseteq N(y) \setminus \{x\}$.
    The other case, not pictured here, is \(R = N[x]\).}
    % $R \setminus \{x, y\} \subseteq N(x) \setminus \{y\}$.}
    \label{fig:R-location-C3C5-free}
\end{figure}

\begin{lemma}
    \label{lemma:MRED_redundant_contained_closed_neighborhood_of_x}
%    Let \(G\) be a \((C_3, C_5)\)-free graph, \(R \in \minred(G)\) be a minimal redundant set, and \(x \in \red(R)\) a redundant vertex.
%    If there exists a neighbor $y$ of $x$ such that  \(y \in \red(R)\), then either \(R = N[y]\) or \(R = N[x]\).
    Let \(G\) be a \((C_3, C_5)\)-free graph, \(R \in \minred(G)\) be a minimal redundant set, and suppose there exists an edge $xy$ such that  
    \(x, y \in \red(R)\).
    %If there exists a neighbor $y$ of $x$ such that  \(y \in \red(R)\), t
    Then either \(R = N[y]\) or \(R = N[x]\).

    \begin{proof}
        By~\Cref{prop:MRED:maximum_distance}, \(R \subseteq N^{2}[x] \cap N^{2}[y]\).
        By \Cref{lem:MRED:bipartite-neighborhood}, $N^{2}[x] \cap N^{2}[y] = N[x] \cup N[y]$, and thus $N^2(x) \cap R = R \cap (N(y)\setminus \{x\})$. By \Cref{lemma:MRED_cases_two_redundant_vertices}, either \(R \cap (N(x) \setminus \{y\}) \neq \emptyset\) or \(R \cap (N(y) \setminus \{x\}) \neq \emptyset\). Without loss of generality, assume we are in the first case.  
        By \Cref{lemma:MRED_cases_two_redundant_vertices} we derive $N^2(x)\cap R=\emptyset$.
        Since $x$ is redundant, we get that $N[x]\subseteq R$, hence that $N[x]=R$ by minimality of $R$.
    \end{proof}
\end{lemma}

We continue by proving properties on $(C_3, C_5, C_6)$-free graphs.

% \subsubsection{Properties in  \texorpdfstring{\((C_3, C_5, C_6)\)}{\{C\_3, C\_5, C\_6\}}-free graphs}
% \label{subsection:MRED-chordal-bipartite}

\begin{lemma}
	\label{lemma:MRED_distinct_neighborhood_distance_2}
    Let \(G\) be a \((C_3, C_5, C_6)\)-free graph, \(R \in \minred(G)\) a minimal redundant set, \(x \in \red(R)\) a redundant vertex, and \(z_{1}, z_2\) be two distinct vertices in \(R \cap N^{2}(x)\).
    Then \((N(z_{1}) \cap N(z_{2})) \setminus N(x) = \emptyset\). 
	
	\begin{proof}
        Let \(z_{1}, z_{2}\) be two distinct vertices in \(R \cap N^{2}(x)\).
        % Note that the assumptions in particular yield \(R \in \minred(G)\). 
        By~\Cref{lemma:MRED_existence_of_privates}, we derive that for each \(i \in [2]\) we have a vertex \(y_{i} \in R\) such that \(y_{i} \in \priv(x, R \setminus \{z_{i}\})\).
        By~\Cref{lem:MRED:bipartite-neighborhood}, the set \(N^2[x]\) induces a bipartite graph.
        % , $z_1,z_2$ and $y_1,y_2$ are non-adjacent.
        Now assuming there exists a vertex \(w \in (N(z_{1}) \cap N(z_{2}))\setminus N(x)\) we obtain an induced cycle of length \(6\), a contradiction.
        This concludes our proof.
	\end{proof}
\end{lemma}

\begin{lemma}
    \label{lemma:MRED_bounds_redundant_vertices_neighborhoods}
    Let \(G\) be a \((C_3, C_5, C_6)\)-free graph.
    If \(R \in \minred(G)\) and \(x \in \red(R)\), then:
    \begin{itemize}
        \item \(|\red(R) \cap N(x)| \leq 2\); and
        \item \(|\red(R) \cap N^{2}(x)| \leq 1\).
    \end{itemize}

    \begin{proof}
        Suppose, for the sake of contradiction, that \(S \coloneqq \{y_{1}, y_{2}, y_{3}\} \subseteq \red(R) \cap N(x)\).
        By~\Cref{lemma:MRED_existence_of_privates}, for each pair of indices \(i,j \in [3]\), where \(i \neq j\), there exists a vertex \(z_{ij}\) such that \(z_{ij} \in \priv(y_{i}, R \setminus \{y_{j}\}) \cap \priv(y_{j}, R \setminus \{y_{i}\})\).
        In particular, $z_{ij}$ is nonadjacent to the remaining vertex in $S\setminus\{y_i,y_j\}$.
        By Lemma \ref{lem:MRED:bipartite-neighborhood}, the set $S$ and the set of all such $z_{ij}$ are independent.
        We conclude that \(G\) contains \(C_{6}\) as an induced subgraph, a contradiction.
        This concludes the proof for the first item.

        Now assume \(|\red(R) \cap N^{2}(x)| \geq 2\) and let \(z_{1}\) and \(z_{2}\) be two distinct vertices in this set.
        Recall that by~\Cref{lemma:MRED_existence_of_privates} there exists a vertex \(w \in \priv(z_{1}, R \setminus \{z_{2}\}) \cap \priv(z_{2}, R \setminus \{z_{1}\})\).
        This, together with the fact that \(N^{2}(x)\) induces an independent set by~\Cref{lem:MRED:bipartite-neighborhood}, implies that \(w \notin N^2[x]\).
        However, this is not possible because by~\Cref{lemma:MRED_distinct_neighborhood_distance_2} the vertices \(z_{1}\) and \(z_{2}\) do not share a common neighbor outside the neighborhood of \(x\).
        This concludes our proof.
    \end{proof}
\end{lemma}

From~\Cref{lemma:MRED_bounds_redundant_vertices_neighborhoods} we conclude that if \(G\) is a \((C_3, C_5, C_6)\)-free graph, then the maximum number of redundant vertices a minimal redundant set contains is 4.
This bound can be further improved by analyzing what happens when there exists a redundant vertex at distance two from another redundant vertex.

\begin{lemma}
    \label{lemma:MRED:redundant_vertex_distance_2}
    Let \(G\) be a \((C_3, C_5, C_6)\)-free graph.
    If \(R \in \minred(G)\) such that \(x \in \red(R)\) and \(|\red(R) \cap N^{2}(x)| = 1\), then \(|R \cap N(x)| = 1\), and so \(|R| = 3\).

    \begin{proof}
        Let \(\red(R) \cap N^{2}(x) = \{z\}\).
        If \(N(z) \subseteq N(x)\), then because \(z\) is redundant, there exists a vertex \(y \in R \cap N(z)\).
        However, since \(x \in R\) the vertex \(z\) is redundant in \(\{x, y, z\}\) and, by minimality, we derive that $R=\{x, y, z\}$ as desired.
        On the other hand, if \(N(z) \not\subseteq N(x)\), then there exists a vertex \(w \in N(z) \setminus N(x)\).
        Due to~\Cref{lem:MRED:bipartite-neighborhood} the vertex \(w\) is at a distance \(3\) of \(x\) and thus it does not belong to the redundant set \(R\) as a consequence of~\Cref{prop:MRED:maximum_distance}.
        In addition, \Cref{lemma:MRED_distinct_neighborhood_distance_2} establishes that no other vertex in \(R \cap N^{2}(x)\) can be adjacent to \(w\), which makes \(w\) a private neighbor of \(z\) and contradicts our assumption.
        This concludes our proof.
    \end{proof}
\end{lemma}

We now conclude the following.

\begin{corollary}\label{cor:MRED:356:number-redundant-vertices}
    If \(G\) is \((C_3, C_5, C_6)\)-free and $R\in \minred(G)$, then $R$ contains at most 3 redundant vertices.
\end{corollary}

The bound in~\Cref{cor:MRED:356:number-redundant-vertices} is tight, as witnessed by 3 consecutive vertices in a cycle of length \(4\).
We end this section with a last property that will be used in the characterization of minimal redundant sets.
% Now that we know the maximum number of redundant vertices in a minimal redundant set of a \((C_3, C_5, C_6)\)-free graph, we can proceed to our characterization of the solutions. 
% We start with the case where \(|\red(R)| = 3\).

\begin{lemma}
    \label{lemma:MRED:two_redundant_neighborhood_of_x}
    Let \(G\) be a \((C_3, C_5, C_6)\)-free graph. 
    If \(R \in \minred(G)\) such that \(x \in \red(R)\) and \(|\red(R) \cap N(x)| = 2\), then \(|R| = 3\).

    \begin{proof}
        Let \(\red(R) \cap N(x) = \{y_{1}, y_{2}\}\) be the pair of redundant vertices adjacent to \(x\).
        We proceed by contradiction.
        Suppose that \(y_{3} \in R \cap N(x) \setminus \{y_{1}, y_{2}\}\).
        It follows from~\Cref{lemma:MRED_existence_of_privates} that for every pair of indices \(i, j \in [3]\), where \(i \neq j\), there exists a vertex \(z_{ij}\) such that \(z_{ij} \in \priv(y_{i}, R \setminus \{y_{j}\})\).
        By~\Cref{lem:MRED:bipartite-neighborhood}, the set \(\{y_{1}, z_{12}, y_{2}, z_{23}, y_{3}, z_{13}\}\) induces a cycle of length \(6\), and we arrive at a contradiction.

        It remains to show that \(R \cap N^{2}(x)\) is empty, and thus suppose that \(z \in R \cap N^{2}(x)\).
        By~\Cref{lem:MRED:bipartite-neighborhood} the set \(N(z) \cap N(y_{i})\) is empty for each \(i \in [2]\), and therefore by~\Cref{lemma:MRED_existence_of_privates} we must have \(z \in \priv(y_{i}, R \setminus \{z\})\).
        This leads to a contradiction, as \(z\) cannot satisfy this property for both vertices, and we conclude our proof.
    \end{proof}
\end{lemma}

\subsubsection{Characterization of redundant sets in graphs with no small cycles}\label{sec:MRED:cycle-free-characterization}

% \subsubsection{Characterization of redundant sets in 
% \texorpdfstring{\((C_3, C_5, C_6)\)}
% {\{C\_3, C\_5, C\_6\}}-free graphs}

In the remainder of this section, we assume $G$ to be a $(C_3, C_5, C_6)$-free graph.
By \Cref{cor:MRED:356:number-redundant-vertices} we know that the number of redundant vertices in a minimal redundant set of $G$ is at most 3.
In the following, we characterize the minimal redundant sets depending on whether their number of redundant vertices is 3, 2, or 1.

Let us recall that any minimal redundant set $R$ contains a redundant vertex $x$, and that by \Cref{lemma:MRED_bounds_redundant_vertices_neighborhoods}, it contains at most two additional redundant vertices at distance 1 from $x$, and at most one at distance 2. 
We derive the following as a corollary of \Cref{lemma:MRED:redundant_vertex_distance_2,lemma:MRED:two_redundant_neighborhood_of_x}.

\begin{corollary}\label{cor:MRED:charac_3_redundants}
    % Let \(G\) be a \((C_3, C_5, C_6)\)-free graph
    If \(R \in \minred(G)\) contains 3 redundant vertices, then \(|R| = 3\).
\end{corollary}

For minimal redundant sets containing two redundant vertices, we derive the following as a corollary of~\Cref{lemma:MRED:redundant_vertex_distance_2,lemma:MRED_redundant_contained_closed_neighborhood_of_x}.

\begin{corollary}\label{cor:MRED:charac_2_redundants}
    % Let \(G\) be a \((C_3, C_5, C_6)\)-free graph
    If \(R \in \minred(G)\) contains 2 redundant vertices, then either:
    \begin{itemize}
        \item \(R = N[x]\) for some \(x \in \red(R)\); or
        \item \(|R| = 3\).
    \end{itemize}
\end{corollary}

We are now left with the characterization of minimal redundant sets containing exactly one redundant vertex which we call \(x\).
We distinguish two cases depending on whether they intersect the neighborhood of $x$ on one or more vertices.
These lemmas are better understood with accompanied Figure~\ref{fig:MRED:one-redundant-two-cases}.

\begin{lemma}
    \label{lemma:MRED:charac_1_redundant_1_neighbor}
    Let \(x \in V(G)\) and \(y \in N(x)\).
    The following are equivalent:
    \begin{enumerate}[\normalfont(a)]
        \item\label{item:MRED:y:a} \(R \in \minred(G)\) where \(\red(R) = \{x\}\) and \(N(x) \cap R = \{y\}\).
        \item\label{item:MRED:y:b} \(R = S \cup \{x, y\}\) where: 
        \begin{enumerate}[\normalfont i)]
            \item\label{item:MRED:y:b:i} \(S \subseteq N^{2}(x)\);
            \item\label{item:MRED:y:b:ii} The vertex \(y\) has a private neighbor with respect to \(S \cup \{x\}\); 
            % There exists a vertex \(z^{\star} \in N(y)\) such that \(z^{\star} \notin S\); \svtodo{$y$ has a private neighbor wrt $S \cup x$?}
            \item\label{item:MRED:y:b:iii} For every \(z \in S\) satisfying \(N(z) \subseteq N(x)\) it follows that \(z \notin N(y)\); and
            \item\label{item:MRED:y:b:iv} \(S\) minimally dominates \(N(x) \setminus \{y\}\);
        \end{enumerate}
    \end{enumerate}

    \begin{proof}
        Suppose that Item \eqref{item:MRED:y:a} holds.
        By~\Cref{prop:MRED:maximum_distance} the set \(R\) is contained in \(N^{2}[x]\), proving Item \eqref{item:MRED:y:b:i}.
        Because \(x\) is the only redundant vertex in \(R\) it follows that \(\priv(y, R)\) is non empty, proving Item \eqref{item:MRED:y:b:ii}.
        % Note in particular that \(N(y) \neq \{x\}\).
        Moreover, we derive that each vertex \(z \in R \setminus \{x, y\}\) such that \(N(z) \subseteq N(x)\) must be self-private, and so in particular \(z \notin N(y)\).
        This proves Item \eqref{item:MRED:y:b:iii}.
        Lastly, \(S = R \setminus \{x, y\}\) minimally dominates \(N(x) \setminus \{y\}\), as if it was not dominating, \(x\) would have a private neighbor, and if it was not minimal with that property, there would exist a set \(S' \subsetneq S\) such that \(S'\) dominates \(N(x) \setminus \{y\}\), a contradiction to the minimality of \(R\).
        This proves Item \eqref{item:MRED:y:b:iv}, hence the first direction.

        Let us now assume that Item \eqref{item:MRED:y:b} holds.
        By Item \eqref{item:MRED:y:b:iv} and the fact that \(y \in R\), the vertex \(x\) is redundant.
        Moreover, by the minimality of $S$, the vertex $x$ gains a private neighbor in any proper subset of $R$ containing $\{x,y\}$.
        The same is true if we remove $y$ from $R$, as $x$ becomes self-private in that case.
        Hence $R$ is minimal with respect to having $x$ redundant, and it remains to argue that $\red(R)=\{x\}$.
        We do this by showing that every element in $R \setminus \{x\}$ has a private neighbor.

        First,
        $y$ has a private neighbor by Item \eqref{item:MRED:y:b:ii}.
        Now, recall that by \Cref{lem:MRED:bipartite-neighborhood} $N^2[x]$ induces a bipartite graph, hence that the set \(S\) is an independent set.
        So the elements of \(S \setminus N(y)\) are self private.
        It remains to consider the vertices \(z \in S \cap N(y)\).
        By Item~\eqref{item:MRED:y:b:iii} the set \(N(z) \setminus N(x)\) is non-empty.
        Moreover, by~\Cref{lemma:MRED_distinct_neighborhood_distance_2} no pair of vertices in \(S\) share a common neighbor outside \(N(x)\), and so every element in the set \(N(z) \setminus N(x)\) is a private neighbor of \(z\).
        We get that $x$ is the only redundant vertex of $R$ as desired.
        This concludes the second direction.
		\end{proof}
\end{lemma}

% \needspace{2cm}

\begin{lemma}
    \label{lemma:MRED:charac_1_redundant_multi_neighbor}
    Let \(x \in V(G)\) and \(Y \subseteq N(x)\) such that \(|Y| \geq 2\).
    The following are equivalent:
    \begin{enumerate}[\normalfont(a)]
        \item\label{item:MRED:Y:a} \(R \in \minred(G)\) where \(\red(R) = \{x\}\) and \(N(x) \cap R = Y\).
        \item\label{item:MRED:Y:b} \(R = S \cup Y \cup \{x\}\) such that: 
        \begin{enumerate}[\normalfont i)]
            \item\label{item:MRED:Y:b:i} \(S \subseteq N^{2}(x)\);
            \item\label{item:MRED:Y:b:ii} \(S \cap N(Y) = \emptyset\);
            \item\label{item:MRED:Y:b:iii} Each vertex \(y \in Y\) has a private neighbor with respect to \(Y \cup \{x\}\);
            % For each \(y \in Y\) the set \(N(y) \setminus N(Y \setminus \{y\})\) is non-empty; and \svtodo{each $y \in Y$ has a private neighbor wrt $Y \cup \{x\}$?}
            \item\label{item:MRED:Y:b:iv} \(S\) minimally dominates \(N(x) \setminus Y\).
        \end{enumerate}
    \end{enumerate}

    \begin{proof}
        Suppose that Item~\eqref{item:MRED:Y:a} holds and let \(S = R \setminus (Y \cup \{x\})\).
        By~\Cref{prop:MRED:maximum_distance} $R\subseteq N^2[x]$ and we obtain Item~\eqref{item:MRED:Y:b:i}.
        Because \(x\) is the only redundant vertex in \(R\), we have $N(y)\neq \{x\}$ for every \(y \in Y\).
        Thus, by \Cref{lemma:MRED_cases_two_redundant_vertices}, since \(|Y| \geq 2\), for every \(y \in Y\) we have that \(R \cap (N(y)\setminus \{x\})=\emptyset\). 
        This establishes Item~\eqref{item:MRED:Y:b:ii}.
        Additionally, because each \(y\) is irredundant and adjacent to \(x\), it follows that \(N(y) \setminus N(Y \setminus \{y\})\) is non-empty and we conclude Item~\eqref{item:MRED:Y:b:iii}.
        %Finally, consider \(S = R \setminus (Y \cup \{x\})\).
        We argue that \(S\) minimally dominates \(N(x) \setminus Y\). 
        If \(S\) was not dominating, then \(x\) would have a private neighbor, contradicting the redundancy of \(x\); and if \(S\) was not minimal, then there would exist a set \(S' \subsetneq S\) that dominates \(N(x) \setminus Y\), and thus \(S' \cup Y \cup \{x\}\) would yield a smaller included redundant set, contradicting the minimality of \(R\).
        This proves Item~\eqref{item:MRED:Y:b:iv} and concludes the first direction.

        Assume that Item~\eqref{item:MRED:Y:b} holds.
        By Item~\eqref{item:MRED:Y:b:iv}, the vertex \(x\) is redundant in the set \(R = S \cup Y \cup \{x\}\).
        This set is minimal with that property as $S$ minimally dominates $N(x)\setminus Y$ by Item~\eqref{item:MRED:Y:b:iv}, and removing any element in $Y$ would provide a private neighbor to $x$, since $Y$ is both an independent set by \Cref{lem:MRED:bipartite-neighborhood}, and nonadjacent to $S$ by \eqref{item:MRED:Y:b:ii}.
        It remains to prove that \(R\) does not contain other redundant vertices, i.e., that every vertex in \(R \setminus \{x\}\) has a private-neighbor.
        Item~\eqref{item:MRED:Y:b:ii} and \Cref{lem:MRED:bipartite-neighborhood} show that the vertices in \(S\) are self-private.
        By Item~\eqref{item:MRED:Y:b:iii} each vertex \(y \in Y\) has a neighbor nonadjacent to vertices in \(Y \setminus \{y\}\), and thus~\Cref{lem:MRED:bipartite-neighborhood} establishes that these neighbors must be private since they cannot be adjacent neither to \(x\) nor to an element in \(S\).
        % We are left with showing that removing any vertex would make \(x\) irredundant.
        % Item~\eqref{item:MRED:Y:b:ii} guarantees that removing any vertex \(y \in Y\) from \(R\) makes \(y\) a private neighbor of \(x\).
        % Since \(S\) is a minimal dominating set of \(N(x) \setminus Y\), deleting any vertex in \(S\) leaves some vertex in \(N(x) \setminus Y\) uncovered; this vertex then becomes a private neighbor of \(x\).
        % Therefore, no proper subset of \(R\) is redundant, and so \(R\) is inclusion-wise minimal.
     \end{proof}
\end{lemma}

\begin{figure}
    \centering
    \includegraphics[scale=0.9]{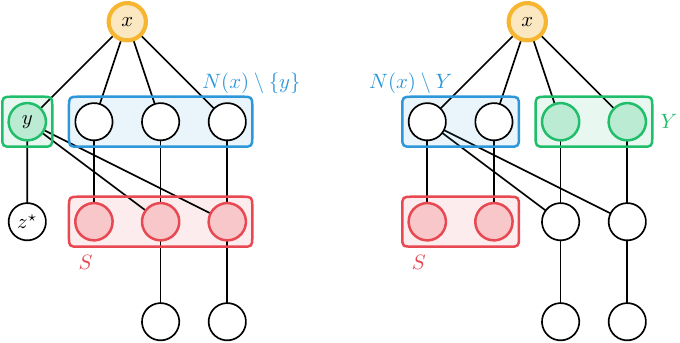}
    \caption{Illustration of Lemma~\ref{lemma:MRED:charac_1_redundant_1_neighbor} on the left, and Lemma~\ref{lemma:MRED:charac_1_redundant_multi_neighbor} on the right.
    Shaded vertices indicate minimal redundant sets, and $x$ (in bold yellow) is the redundant vertex in each case. 
    On the left, $y$ (boxed in green) has a private vertex as well as each vertex of $S$ adjacent to $y$ (the vertices of $S$ are boxed in red) and $S$ minimally dominates $N(x) \setminus \{y\}$ (boxed in blue).
    On the right, $Y$ has at least two elements, $S$ contains no vertex adjacent to $Y$ and minimally dominates $N(x) \setminus Y$.}
    \label{fig:MRED:one-redundant-two-cases}
\end{figure}

\subsubsection{Algorithm for graphs with no small cycles}\label{sec:MRED:cycle-free-algorithm}

%\svtodo[inline]{working on the presentation of this part (not proofs). Oscar (blessed be he) made a back up of the section.}

We now propose an algorithm for listing the minimal redundant sets of a graph with no small cycles.
It runs with polynomial-delay for $(C_3, C_5, C_6, C_8)$-free graphs and incremental quasi-polynomial time for $(C_3, C_5, C_6)$-graphs.

The algorithm follows the breakdown of solutions based on redundant vertices we developed so far.
First, following Corollaries~\ref{cor:MRED:charac_3_redundants} and \ref{cor:MRED:charac_2_redundants}, it enumerates all minimal redundant sets with $2$ or $3$ redundant vertices by running through all possible closed neighborhoods and triplets of vertices, and checking for each of these sets if it is a minimal redundant, all of which in polynomial time. 
Then it lists, for each vertex $x$, the minimal redundant sets whose unique redundant vertex is $x$.
According to Lemma~\ref{lemma:MRED:charac_1_redundant_1_neighbor} and Lemma~\ref{lemma:MRED:charac_1_redundant_multi_neighbor}, these can be further partitioned with respect to the number of neighbors of $x$ in the redundant sets.
We detail our approaches for both cases below.

Let us first consider the computation of all minimal redundant sets containing exactly one neighbor of $x$.
The algorithm first ranges over $N(x)$ and for each choice $y \in N(x)$, selects a neighbor $z^\star$ of $y$ that will be kept as a private neighbor of $y$. 
For simplicity, let us put $Z_y \coloneqq \{z \mid z \in N(y),\ N(z) \subseteq N(x)\}$.
Then, if any, we list the minimal dominating sets $S \subseteq N^2(x) \setminus (\{z^\star\} \cup Z_y)$ of $N(x) \setminus \{y\}$.
According to Lemma~\ref{lemma:MRED:charac_1_redundant_1_neighbor}, the sets $R=S\cup \{x,y\}$ obtained that way are precisely the minimal redundant sets such that $\red(R) = \{x\}$, $N(x) \cap R = \{y\}$. 
% Note that $z^\star \notin R$.
% and $z^\star \notin R$ are precisely the sets $S \cup \{x, y\}$ where $S \subseteq N^2(x) \setminus (\{z^\star\} \cup Z_y)$ is a minimal dominating set of $N(x) \setminus \{y\}$, if there is any.
Also, note that the sets $S$ are the solutions to an instance of red-blue domination where we intend to minimally dominate $N(x) \setminus \{y\}$ (blue) with vertices from $ N^2(x) \setminus (\{z^\star\} \cup Z_y)$ (red).

As for the complexity, the pairs $\{y, z^\star\}$ can be listed in polynomial time, much as the set $N^2(x) \setminus (\{z^\star\} \cup Z_y)$ for each pair.
However, given some $y$, different $z^\star$'s can produce repetitions, as these are chosen to be excluded from the resulting solutions and hence can lead to same sets $S$.
A given solution may thus be repeated $O(n)$ times.

In general graphs red-blue domination can be solved in incremental quasi-polynomial time~\cite{fredman1996complexity,golovach2016enumerating}.
When applying the algorithm for red-blue domination on each $z^\star$ successively, the delay before the next solution thus remains bounded by the number of solutions already obtained times a polynomial factor (the number of repetitions).
We obtain:

\begin{proposition}\label{prop:C3C5C6-algo}
Let $G$ be a $(C_3, C_5, C_6)$-free graph and let $x \in V(G)$.
There is an incremental-quasi-polynomial time algorithm that lists all minimal redundant sets $R$ where $\red(R) = \{x\}$ and $R \cap N(x) = \{y\}$ for some $y \in N(x)$. 
\end{proposition}

In $(C_6, C_8)$-free bipartite graphs, red-blue domination can be solved with polynomial delay~\cite{kante2018enumerating} when red and blue vertices belong to distinct parts. 
However, our approach does not directly preserve delay due to repetitions.
%Hence, using the red-blue domination algorithm on each possible choice of $z^*$ may not preserve delay due to repetitions.
Yet, we can still preserve polynomial delay overall.
To do so, instead of outputting a previously unseen solution, we store it into a queue, a member of which is popped and printed every $O(n) \cdot f(n)$ time, where $f(n)$ is the delay for red-blue domination.
Storing solutions into the queue and keeping a structure to track repetitions requires however exponential space.
We get:
\begin{proposition}\label{prop:C3C5C6C8-algo}
Let $G$ be a $(C_3, C_5, C_6, C_8)$-free graph and let $x \in V(G)$.
There is a polynomial delay algorithm that lists all minimal redundant sets $R$ where $\red(R) = \{x\}$ and $R \cap N(x) = \{y\}$ for some $y \in N(x)$. 
\end{proposition}

We now turn our attention to minimal redundant sets containing at least two neighbors of $x$, corresponding to Lemma~\ref{lemma:MRED:charac_1_redundant_multi_neighbor}.
Two steps are needed.
The first step consists in identifying subsets of $N(x)$ that can be extended into a minimal redundant sets. 
Given $Y \subseteq N(x)$, $|Y| \geq 2$, we say that $Y$ is \emph{extendable} with respect to $x$ if there exists a minimal redundant set $R$ with $\red(R) = \{x\}$ and $R \cap N(x) = Y$.
Below, we characterize extendable sets and show that they constitute an independence set system (i.e., a family of sets containing the empty set and being closed by subset), disregarding singletons and the empty set.

\begin{lemma} \label{lem:MRED:extendable_characterization}
    Let \(x \in V(G)\) be a vertex and \(Y \subseteq N(x)\) be a set such that \(|Y| \geq 2\).
    Then, the set \(Y\) is extendable with respect to \(x\) if and only if the two conditions hold:
    \begin{enumerate}[\normalfont (a)]
        \item\label{item:MRED:extendable:a} Each $y \in Y$ has a private neighbor with respect to $Y \cup \{x\}$; and
        \item\label{item:MRED:extendable:b} The set $N^2(x) \setminus N(Y)$ dominates $N(x) \setminus Y$.
    \end{enumerate}

    \begin{proof}
        Suppose \(Y\) is extendable.
        Then, there is a set \(R \in \minred(G)\) such that \(R \cap N(x) = Y\) and \(x \in \red(R)\) is the unique redundant vertex.
        
        By~\Cref{lemma:MRED:charac_1_redundant_multi_neighbor}, the set \(R\) can be decomposed into the union \(R = S \cup Y \cup \{x\}\) such that \(S\) and \(Y\) satisfy Conditions~\eqref{item:MRED:Y:b:i}--~\eqref{item:MRED:Y:b:iv} of \Cref{lemma:MRED:charac_1_redundant_multi_neighbor}. By Conditions~\eqref{item:MRED:Y:b:i}, \eqref{item:MRED:Y:b:ii}, and \eqref{item:MRED:Y:b:iv}, \(S \subseteq N^{2}(x) \setminus N(Y)\) minimally dominates \(N(x) \setminus Y\), and thus Item~\eqref{item:MRED:extendable:b} holds. By Condition~\eqref{item:MRED:Y:b:iii},
        each vertex in \(Y\) has a private neighbor with respect to \(Y \cup \{x\}\), and we conclude Item~\eqref{item:MRED:extendable:a} holds, hence completing the first direction of the proof.
        
        % By definition of \(R\), every vertex \(y \in Y\) has a private neighbor---that is, the set \(\priv(y, R)\) is non-empty.
        % Consequently, Item~\eqref{item:MRED:extendable:a} holds.
        % Since \(x\) is redundant, its neighborhood \(N(x)\) is dominated.
        % Moreover, from the equality \(R \cap N(x) = Y\) no vertex in \(N(x) \setminus Y\) belongs to \(R\).
        % Recall that~\Cref{lem:MRED:bipartite-neighborhood} implies that \(N(x)\) is an independent set and that~\Cref{lemma:MRED_cases_two_redundant_vertices} ensures that \(N(Y) \cap R\) is empty.
        % Therefore, each vertex in \(N(x) \setminus Y\) is dominated by a vertex in \(N^{2}(x) \setminus N(Y)\).
        % Hence Item~\eqref{item:MRED:extendable:b} holds, completing the first direction of the proof.

        Assume that both conditions~\eqref{item:MRED:extendable:a} and~\eqref{item:MRED:extendable:b} hold.
        Item~\eqref{item:MRED:extendable:b} establishes that \(N^{2}(x) \setminus N(Y)\) dominates \(N(x) \setminus Y\).
        Consequently, there exists an inclusion-wise minimal set \(D \subseteq N^{2}(x) \setminus N(Y)\) that dominates \(N(x) \setminus Y\).
        Let \(R = Y \cup D \cup \{x\}\).
        
        Notice that \(D \cap N(Y)\) is empty; otherwise, there would exist a vertex \(y \in Y\) such that its removal from \(R\) would maintain \(x\) redundant by our assumption on the cardinality of \(Y\).
        Therefore, since \(|Y| \geq 2\) and \(D \cap N(Y) = \emptyset\), we obtain by~\Cref{lemma:MRED:charac_1_redundant_multi_neighbor} that the set \(R\) is a minimal redundant set that satisfies the desired properties.
        This concludes our proof.
        % By Item~\eqref{item:MRED:extendable:a} each vertex in \(Y\) has a private neighbor with respect to \(R\).
        % Furthermore, the definition of \(D\) tells us that every vertex in \(D\) are self-private.
        % Thus, the vertex \(x\) is the only redundant set in \(R\).
        % It remains to show the set is inclusion-wise minimal.
        % Due to the minimality of \(D\), removing any vertex from \(D\) would turn a vertex in \(N(x) \setminus Y\) into a private neighbor of \(x\); and because \(D \cap N(Y)\) is empty, removing a vertex \(y \in Y\) from the set makes \(y\) a private neighbor of \(x\).
        % Hence, the set \(R\) is inclusion-wise minimal and we conclude our proof.
    \end{proof}
\end{lemma}

\begin{lemma}
    \label{lemma:MRED_x-extendable_sets}
    Let \(x \in V(G)\) be a vertex.
    If \(Y^{\star} \subseteq N(x)\) is extendable with respect to \(x\), then every subset \(Y \subsetneq Y^{\star}\) satisfying \(|Y| \geq 2\) is also extendable with respect to \(x\).

    \begin{proof}
        Let \(Y^{\star}\) be an extendable set and \(Y \subsetneq Y^{\star}\) be a proper non-empty subset of \(Y^{\star}\).
        By~\Cref{lem:MRED:extendable_characterization}, each vertex \(y \in Y^{\star}\) has a private neighbor \(z_{y} \in N(y)\).
        Consequently, the vertices in \(Y^{\star} \setminus Y\) can be dominated by these private neighbors, hence by \(N^{2}(x) \setminus N(Y)\). By~\Cref{lem:MRED:extendable_characterization} since $|Y|\geq 2$ we conclude that \(Y\) is an extendable set.
    \end{proof}
\end{lemma}

Let $Y$ be an extendable set.
Following Lemmas~\ref{lemma:MRED:charac_1_redundant_multi_neighbor} and \ref{lem:MRED:extendable_characterization}, the minimal redundant sets $R$ such that $\red(R) = \{x\}$ and $R \cap N(x) = Y$ are precisely those of the form $S \cup Y \cup \{x\}$ where $Y$ is extendable and $S \subseteq N^2(x) \setminus N(Y)$ minimally dominates $N^2(x) \setminus N(Y)$.
Again, this can be framed as red-blue domination: we need to minimally dominate $N(x) \setminus Y$ (blue) with vertices from $N^2(x) \setminus N(Y)$ (red).
Therefore, the algorithm for this part lists all extendable sets $Y$ using an algorithm to enumerate all members of an independence set system, discarding singletons and the empty set, and for each such $Y$ solves the corresponding instance of red-blue domination.
Note that, as long as the algorithm for red-blue domination does not produce repetitions, this algorithm will not either, as extendable $Y$'s partition the solutions to be enumerated.

As for complexity, observe first that whether a set is extendable can be tested in polynomial time.
Given that the members of an independence set system can be enumerated with polynomial delay provided they can be recognized in polynomial time (see, e.g., \cite{kante2014enumeration} for an example), we derive the following propositions, mimicking Propositions \ref{prop:C3C5C6-algo} and \ref{prop:C3C5C6C8-algo}:

\begin{proposition}
Let $G$ be a $(C_3, C_5, C_6)$-free graph and let $x \in V(G)$.
There is an incremental-quasi-polynomial time algorithm that lists all minimal redundant sets $R$ where $\red(R) = \{x\}$ and $R \cap N(x) = Y$ for some $Y \subseteq N(x)$ and $|Y| \geq 2$.
\end{proposition}

\begin{proposition}
Let $G$ be a $(C_3, C_5, C_6, C_8)$-free graph and let $x \in V(G)$.
There is a polynomial-delay algorithm that lists all minimal redundant sets $R$ where $\red(R) = \{x\}$ and $R \cap N(x) = Y$ for some $Y \subseteq N(x)$ and $|Y| \geq 2$. 
\end{proposition}

\iflongelse{Combining the above propositions we obtain~\Cref{thm:MIRR:C3C5C6C8} that we restate below.}{Combining the above propositions we obtain~\Cref{thm:MIRR:C3C5C6C8}.}
Let us observe that a solution of the form $N[x]$ might be obtained twice: at preprocessing or in the algorithm used for finding all minimal redundant sets where $x$ is the unique redundant vertex. 
Since the repetition is unique, this, however, can be handled without impact on the delay by simply not outputting the solution the second time it is produced.

\iflongelse{\thmMREDceightfree*}{}

\subsection{Co-bipartite graphs}
\label{subsection:MRED_co-bipartite}

In this section, we prove that, unless $\P = \NP$, there is no output-polynomial time algorithm solving \graphmredenum{} in co-bipartite graphs.
%\graphmredenum{} \sv{does not admit} an output-polynomial time algorithm when restricted to co-bipartite graphs, then $\P = \NP$.
Our reduction relies on earlier results of Boros and Makino~\cite{boros2024generating}, stating that \hypmredenum{} is intractable for hypergraphs of dimension 3.
Nevertheless, the problem can be solved in polynomial time for hypergraphs of maximum degree 3 due to every solution having size at most 4.

We consider a hypergraph $\H$ of dimension 3, whose transposed hypergraph $\H^t$ has maximum degree at most 3.
By looking at the co-bipartite incidence graph $C(\H)$ of $\H$, we prove that listing the members of $\minred(C(\H))$ amounts to enumerating $\minred(\H)$, together with a polynomial number of solutions of size at most 4---including $\minred(\H^t)$ and solutions spread across the parts of $C(\H)$.

\iflongelse{}{This yields the following.}

\begin{theorem} \label{theorem:MRED_conp_complete_cobipartite}
The problem \graphmredenum{} restricted to co-bipartite graphs cannot be solved in output-polynomial time unless $\P = \NP$. 
\end{theorem}

\begin{proof}
Let $\H$ be a hypergraph of dimension 3 and consider its co-bipartite incidence graph $C \coloneqq C(\H)$.
Recall that the vertex set of \(C\) is partitioned into sets \(V\) and \(U\) corresponding to the vertices and edges of \(\H\), respectively.
If $e$ is an element of $U$ by $E$ we mean its corresponding hyperedge.
%with vertices $V = \{v_1, \dots, v_n\}$ and edges $\mathcal{E} = \{E_1, \dots, E_m\}$, and consider its cobipartite incidence graph $C \coloneqq C(\H)$.
%For the recall, $C(\H)$ has parts $V$ and $U = \{e_1, \dots, e_m\}$ where each $e_i$ corresponds to $E_i$.
We first relate the minimal redundant sets of $C(\H)$ with the minimal redundant sets of $\H$.
Namely, we show the following equality:
\begin{align*}
\minred(C)  & = \mathcal{X} \cup \mathcal{R} \quad \text{ with } \\ 
\mathcal{X} & = \{R \mid R \in \minred(C),\ |R| \leq 4 \} \\ 
\mathcal{R} & = \{R \mid R \in \minred(\H),\ |R| \geq 5 \}
\end{align*}

We start by proving the inclusion $\minred(C) \subseteq \mathcal{X} \cup \mathcal{R}$.
Let $R \in \minred(C)$. 
Let us further assume that $|R| \geq 5$, as the inclusion trivially holds for smaller $R$.
Note that for any $v \in V$, $e \in U$, $ve$ is dominating $C$.
Therefore, $v, e \in R$ would imply that $|R| \leq 3$.
Hence, either $R \subseteq U$ or $R \subseteq V$.
We prove that $R \subseteq V$.
Given that $|R| > 4$ and $U, V$ are cliques, no vertex of $R$ is self-private.
Now, since $\H^t$ is of degree at most 3, each vertex in $U$ has at most 3 neighbors in $V$.
Hence to make a given $e \in U$ redundant, that is to dominate $N[e]$, at most 3 vertices are needed.
%.Hence, to make a given vertex $e$ of $U$ redundant, at most 3 other vertices of $U$ are needed to dominate $N[e]$.
It follows that a minimal redundant set of $C$ included in $U$ is of size at most 4.
We deduce that $R \subseteq V$ as expected.

Now since $V$ is a clique and $|R| \geq 5$, no vertex $x$ in $R$ gains a private neighbor in $V$ when removing another vertex from $R$.
Consequently $R$ is minimal with the property that not all its elements have private neighbors in $U$.
As $ve$ is an edge of $C$ if and only if $v \in E$ in $\H$, we derive that $R$ is a minimal redundant set of $\H$.

We now turn to the other inclusion.
Let $R \in \minred(\H)$ with $|R| \geq 5$. 
Then $R\subseteq V$.
Again, note that no element of $V$ can be a private neighbor of a vertex in $R$, nor of a vertex in $R\setminus\{x\}$ for any $x\in R$.
As $ve$ is an edge of $C$ if and only if $v \in E$ in $\H$, we derive that $R$ is a minimal redundant set of $C$.

We are ready to conclude.
We show that an output-polynomial time algorithm $\algoA$ for solving \graphmredenum{} in co-bipartite graphs can be used to produce an output-polynomial time algorithm $\algoB$ solving \hypmredenum{} in hypergraph of dimension 3.
This would entail $\P = \NP$ due to \cite{boros2024generating}.
The algorithm $\algoB$ first runs over all subsets of $V$ of size $\leq 4$ and identifies among them the minimal redundant sets of $\H$ with size $\leq 4$.
Then, it builds the graph $C$ in polynomial time and runs the algorithm $\algoA$ on input $C$.
Based on the previous equality, $\algoA$ will indeed produce all remaining minimal redundant sets of $\H$.
Given that $C$ is of polynomial size in the size of $\H$, and $|\minred(C)| \leq |\minred(\H)| + (|V(\H)+|\E(\H)|)^4$, the algorithm $\algoA$ will run in a time polynomially bounded by the sizes of $|\H|$ and $|\minred(\H)|$.
This completes the description of $\algoB$, which also runs in output-polynomial time as expected, which concludes the proof.
\end{proof}

\iflongelse{Note that our reduction preserves the delay as there is only a polynomial number of solutions to discard.}{}

\iflongelse{
    \subsection{Split graphs}
    \label{subsection:MRED_split}
    
    In this section, we show that \graphmredenum{} remains intractable even when restricted to split graphs.
    This result contrasts with \graphmirrenum{}, as the latter can be solved with polynomial delay and space; see Theorem~\ref{thm:MIRR:split}.
    The reduction we use is the split incidence graph $S_2(\H)$ of an arbitrary hypergraph $\H$.
    Recall that $S_2(\H)$ is obtained from $B(\H)$ by completing $V$ into a clique.
    %We give an example in Figure~\ref{fig:split-mred}.
    We then argue that enumerating the minimal redundant sets of $S_2(\H)$ amount to list the minimal redundant sets of $\H$ plus a polynomial number of solutions of size $2$, following similar arguments as in the proof of Theorem~\ref{theorem:MRED_conp_complete_cobipartite}.
    
    % \begin{figure}
    %     \centering
    %     \includegraphics[scale=0.9]{images/split-mred.pdf}
    %     \caption{The reduction used in Theorem~\ref{theorem:MRED_conp_complete_split}. On the left a hypergraph $\H$. On the right the associated incidence split $S_2(\H)$ where $V$ is completed into a clique, illustrated by a grey dotted zone.
    %     Vertices emphasized in bold red form a minimal redundant set of both $\H$ and $S_2(\H)$.}
    %     \label{fig:split-mred}
    % \end{figure}
    
    \begin{theorem} \label{theorem:MRED_conp_complete_split}
    The problem \graphmredenum{} restricted to split graphs cannot be solved in output-polynomial time unless $\P = \NP$.
    \end{theorem}
    
    \begin{proof}
    Let $\H$ be a hypergraph with vertices $V = \{v_1, \dots, v_n\}$ and edges $\E = \{E_1, \dots, E_m\}$, and consider the split incidence graph $S \coloneqq S_2(\H)$. 
    Recall that its vertex set is partitioned into $V$ and $U$ where $V$ induces a clique and $U = \{e_1, \dots, e_m\}$ induces an independent set where $e_j$ represents $E_j$.
    %Moreover, the set \(V\) induces a clique.
    We first characterize $\minred(S)$ in terms of $\minred(\H)$:
    \begin{align*}
    \minred(S) & = \mathcal{X} \cup \mathcal{R} \quad \text{ with } \\ 
    \mathcal{X} & = \{R \mid R \in \minred(S),\ |R| = 2\} \\ 
    \mathcal{R} & = \{R \mid R \in \minred(\H),\ |R| \geq 3\}
    \end{align*}
    
    We first prove the inclusion $\minred(S) \subseteq \mathcal{X} \cup \mathcal{R}$.
    Let $R \in \minred(S)$. 
    Let us further assume that $|R| \geq 3$, as the inclusion trivially holds for $|R|=2$, and no minimal redundant set has less than two elements.
    Observe that for any $v \in V$ and any $e \in U$, if $v$ is adjacent to $e$ then $ve$ is a minimal redundant set of size 2 as $N[e] \subseteq N[v]$.
    Thus $R$ does not contain any such edge.
    Let $v$ be a redundant vertex of $R$, which, by the previous argument, must thus lie in $V$.
    As $N[v]$ is dominated by $R \setminus \{v\}$, but $R\cap U\cap N(v)=\emptyset$, there must exist $w \in V\cap R$ distinct from $v$ that prevents $v$ from being self-private. 
    Suppose toward a contradiction that there exists $e\in U\cap R$.
    Since $R\cap U\cap N(v)=\emptyset$, $N[e]\cap N[v] \subseteq N[w]\cap N[v]$ and so $N[v]$ remains dominated by $R \setminus \{e\}$, contradicting the minimality of $R$.
    So $R \subseteq V$.
    
    Now since $V$ is a clique and $|R| \geq 3$, no vertex $x$ in $R$ gains a private neighbor in $V$ when removing another vertex from $R$.
    Consequently $R$ is minimal with the property that not all its elements have private neighbors in $U$.
    As $ve$ is an edge of $S$ if and only if $v \in E$ in $\H$, we derive that $R$ is a minimal redundant set of $\H$.
    
    We now turn to the other inclusion.
    Let $R \in \minred(\H)$ with $|R| \geq 3$. 
    Then $R\subseteq V$.
    Again, note that no element of $V$ can be a private neighbor of a vertex in $R$, nor of a vertex in $R\setminus\{x\}$ for any $x\in R$.
    As $ve$ is an edge of $S$ if and only if $v \in E$ in $\H$, we derive that $R$ is a minimal redundant set of $S$.
    
    We are ready to conclude.
    We prove that an output-polynomial time algorithm $\algoA$ for solving \graphmredenum{} in split graphs yields an output-polynomial time algorithm $\algoB$ solving \hypmredenum{}.
    This entails $\P = \NP$ due to \cite{boros2024generating}.
    The algorithm $\algoB$ first identifies the minimal redundant sets of $\H$ of size at most $2$ by ranging over singleton and pairs of vertices of $V$.
    Then, it builds the graph $S$ in polynomial time and runs the algorithm $\algoA$ on input $S$.
    Based on the previous equality, $\algoA$ will indeed produce all remaining minimal redundant sets of $\H$ in time $|S| + |\minred(S)|$.
    Since $|S| + |\minred(S)| \leq |\minred(\H)| + |V|^2$, we obtain that $\algoB$ runs in output-polynomial time as expected, which concludes the proof.
    \end{proof}
}{}

\section{Perspectives}\label{sec:conclusion}

In this paper, we studied the problems \graphmirrenum{} and \graphmredenum{} in graph classes capturing incidence relations such as co-bipartite, split and bipartite graphs.
In particular, we proved the hardness of \graphmredenum{} on split and co-bipartite graphs, while for \graphmirrenum{} we showed that the problem is as hard in general graphs as in co-bipartite graphs.
\Cref{qu:output-poly-for-graphs} remains thus open for \graphmirrenum{}.

Concerning tractability, in addition to the case of \graphmirrenum{} admitting a polynomial-delay algorithm on split graphs, originally claimed by Uno in \cite{bodlaender2015open}, we showed that both problems admit polynomial-delay algorithms on generalizations of chordal bipartite graphs.
The case of bipartite graphs is left open, which leads to the following question:

\begin{question}
What is the complexity of \graphmirrenum{} and \graphmredenum{} in bipartite graphs?
\end{question}

\iflongelse{
    Towards this question, we conclude this paper by giving evidence that the algorithms we proposed do not straightforwardly apply to the bipartite case.
    First, we prove that the sequential method we applied to solve \graphmirrenum{} in strongly orderable graphs cannot be used on bipartite graphs.

    \begin{figure}
        \centering
        \includegraphics[scale=0.9]{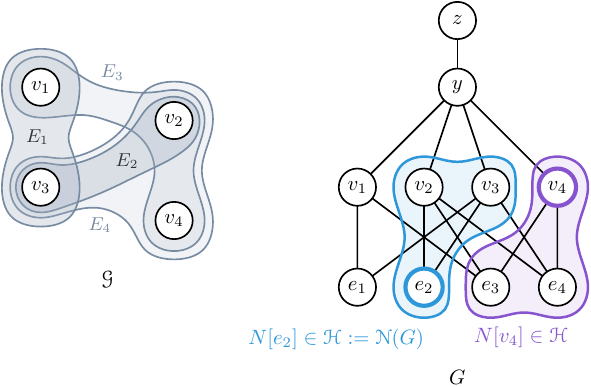}
        \caption{The reduction of Theorem~\ref{theorem:MIRR_npcomplete_irp_trace_bipartite}. We start from a hypergraph $\G$ (on the left), build its bipartite incidence graph $B(\G)$ to which we add a vertex $y$ universal to $V$ and a vertex $z$ pending at $y$. The resulting graph $G$ is pictured on the right. Then, we consider the closed neighborhood hypergraph $\H \coloneqq \N(G)$. 
        Two of its edges, $N[e_2]$ and $N[v_4]$, are drawn on $G$.}
        \label{fig:bip-mirr-orderedgen}
    \end{figure}
}{
    Towards this question, we observe that the sequential method we applied to solve \graphmirrenum{} in strongly orderable graphs cannot be used straightforwardly in the bipartite case. Namely, we prove the following.
}

\begin{theorem}
    \label{theorem:MIRR_npcomplete_irp_trace_bipartite}
    Unless $\P = \NP$, there is no output-polynomial time algorithm which, given a bipartite graph $G$, an ordering $v_1,\dots,v_n$ of its vertices, an integer \(i \in [n-1]\), and \(I^{\star} \in \maxirr(\H_{i})\) where $\H=\N(G)$, enumerates \(\children(I^{\star}, i)\).
    % If there is an output-polynomial time algorithm which, given a bipartite graph $G$, an ordering $v_1,\dots,v_n$ of its vertices, an integer \(i \in [n-1]\), and \(I^{\star} \in \maxirr(\H_{i})\) where $\H=\N(G)$, enumerates \(\children(I^{\star}, i)\), then \(\P = \NP\).
    
	\begin{proof}
		Let \(\G\) be a instance of \hypmirrenum{} in hypergraphs, for which the existence of an output-polynomial time algorithm implies \(\P = \NP\)~\cite[Theorem 2]{boros2024generating}.
        Let us assume without loss of generality that every vertex of $\G$ appears in at least one hyperedge, that $\G$ does not contain the empty hyperedge, and let us denote by $N$ and $M$ the number of vertices and edges of $\G$, respectively.
        Let us finally assume that $\G$ is not a trivial instance in which $V(\G)$ is irredundant. 

        Now, let $G$ be the bipartite incidence graph $B(\G)$ of $\G$ extended with a vertex \(y\) adjacent to every vertex in \(V=V(\G)\), and a vertex \(z\) adjacent only to \(y\).
        %We construct the incidence bipartite graph \(G\) with partitions \((X, \mathcal{E})\) representing the vertices and hyperedges of \(\G\), respectively, add a vertex \(y\) adjacent to every vertex in \(X\), and a vertex \(z\) adjacent only to \(y\). \odtodo{use the notation $B(\G)$ for the bip of incidence}
        Lastly, we consider the ordering \(e_{1}, \dots, e_{M}, v_{1}, \dots, v_{N}, y, z\) of the vertices of \(G\), and set \(\H \coloneqq \N(G)\), $I^{\star} \coloneqq V$, and $i \coloneqq N+M$.
        This concludes the construction of the instance.
        It is illustrated on an example in Figure~\ref{fig:bip-mirr-orderedgen}.
        
        The set \(V\) is a maximal irredundant set of \(\H_{i}\), because every vertex in \(V\) is self-private and every other vertex in $V_i$ is dominated. 
        Notice that by our assumptions \(V \cup \{y\}\) is redundant, and that \(y\) has a private hyperedge in \(\H_{i+1}\), namely the trace $\{y\}$ of $N[z]$.
        Since the vertices in \(V\) cannot be self-private, the extensions of \(V\) are of the form \(Y \cup \{y\}\) where $Y$ is a maximal subset of $V$ having private neighbors in $\E$. 
        Note that the children of $I^{\star}=V$ with respect to $i$ are precisely the family of all such extensions.
        Moreover, notice that such sets \(Y\) are the maximal irredundant sets of \(\G\); otherwise, we can add another vertex from \(V \setminus Y\) into \(Y\), contradicting its maximality.
        As the construction can be conducted in polynomial time, we get the desired result.
	\end{proof}
\end{theorem}

As for \graphmredenum{}, the strategy we used in graphs without small cycles relied on two crucial facts: (1) a minimal redundant set has a constant number of redundant vertices (at most $3$), and (2) the minimal redundant sets containing a single redundant vertex and a unique neighbor of this vertex can be enumerated efficiently.
In bipartite graphs these two facts do not hold anymore\iflongelse{, as we show now.}{.}

For redundant sets with unbounded number of redundant vertices, let us consider the following bipartite graph $G$.
We start from vertices $v_1, \dots, v_n$ and create, for each $i \neq j$ a new vertex $u_{ij}$ adjacent to $v_i$ and $v_j$.
Then, we add a vertex $y$ adjacent to each $v_i$ and a vertex $z$ adjacent to $y$.
%\jctodo{Cannot we remove $z$ from the construction? I believe we will have $\red(R) = R$ in this case.}\odtodo{I agree but Simon had in mind something like "let's make this example not fall in the easy characterizations of redundants vertices being a single neighborhood". I guess this remark should be added or $z$ removed.}
This complete the description of $G$.
The set $R \coloneqq \{y, v_1, \dots, v_n\}$ is a minimal redundant set of $G$ where $\red(R) = \{v_1, \dots, v_n\}$ is indeed of unbounded size.
\iflongelse{We illustrate this construction in Figure~\ref{fig:bip-multiple-red}.

\begin{figure}[t]
    \centering
    \includegraphics[scale=0.8]{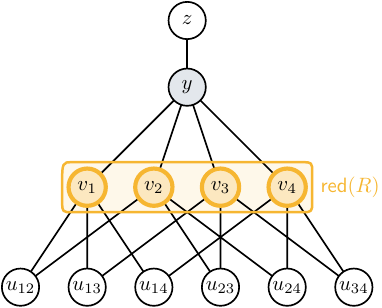}
    \caption{A bipartite graph with a minimal redundant set (shaded vertices) having an unbounded number of redundant vertices (boxed in yellow). The vertex $z$ shows that a minimal redundant set with an unbounded number of redundant vertices does not need to be the closed-neighborhood of a vertex.}
    \label{fig:bip-multiple-red}
\end{figure}
}{}

For the task of enumerating the minimal redundant sets with a single redundant vertex $x$ and containing a unique neighbor $y$ of $x$, i.e., the case of Lemma~\ref{lemma:MRED:charac_1_redundant_1_neighbor}, we show in the next theorem that it  becomes intractable in bipartite graphs.
The hardness comes from the fact that, when selecting a subset of $N(y)$ to be part of the redundant set, one must guarantee that each chosen vertex has a private neighbor outside $N^2(x)$ all the while being useful to dominate $N(x) \setminus \{y\}$, a problem that does not arise in bipartite graphs with no induced $C_6$.

\begin{theorem}
    \label{thm:MRED:method_limitation}
    Let $G$ be a bipartite graph and $x, y \in V(G)$ be a pair of adjacent vertices.
    Then, unless $\P = \NP$, there is no output-polynomial time algorithm to enumerate the minimal redundant sets $R$ of $G$ that satisfy $\red(R) = \{x\}$ and $R \cap N(x) = \{y\}$. 
\end{theorem}

\begin{proof}
Let \(\phi\) be an instance of \(3\)-SAT with literals \(X = \{x_{1}, \overline{x}_{1}, \dots, x_{n}, \overline{x}_{n}\}\) and clauses \(C \coloneqq \{c_{1}, \dots c_{m}\}\).
We construct the incidence bipartite graph \(G\) of \(\phi\) with partition \((X, C)\).
Moreover, we add two adjacent vertices \(x\) and \(y\), with $x$ complete to $C$, $y$ complete to $X$, 
% that are made complete respectively to \(C\) and \(X\), 
and a vertex \(u\) whose only neighbor is \(y\).
Finally, for each \(i \in [n]\), we create a new vertex \(y_{i}\) adjacent to both \(x_{i}\) and \(\overline{x}_{i}\).
This finishes our construction of $G$.
An illustration of the construction is depicted in~\Cref{figure:MRED_bipartite_hardness_construction}.

We first prove that there exists a minimal redundant set $R$ of $G$ such that $\red(R) = \{x\}$ and $R \cap N(y) = \{x\}$ if and only if $\phi$ has a satisfying truth assignment.
We start with the if part.
Let $A$ be the set of literals set to true in a minimal partial satisfying assignment of $\phi$ and let $R := A \cup \{x, y\}$.
We show that $R$ satisfies the aforementioned requirements.
We readily have that $R \cap N(x) = \{y\}$ by construction of $G$.
Since $A$ is a partial satisfying assignment of $\phi$, each vertex in $C$ is dominated by a vertex of $A$.
Therefore, $R$ indeed dominates $N[x]$, making $x$ redundant.
In $R$, $y$ has private neighbor $u$ and any literal $\ell_i$ has private neighbor $y_i$ as $R \subseteq X \cup \{x\}$ and $A$ cannot contain both $x_i, \bar{x_i}$ for any $i \in [n]$ by assumption.
We thus obtain that $\red(R) = \{x\}$.
To see that $R$ is minimal, observe that discarding $y$ would make $x$ self-private and that discarding any literal of $A$ would give a private neighbor $c_j$ to $x$ by assumption on the minimality of $A$ for satisfying $\phi$.
As all other vertices of $R$ are irredundant, we deduce that $R$ is indeed minimally redundant.

\begin{figure}[ht!]
    \centering
    \includegraphics[scale=0.8]{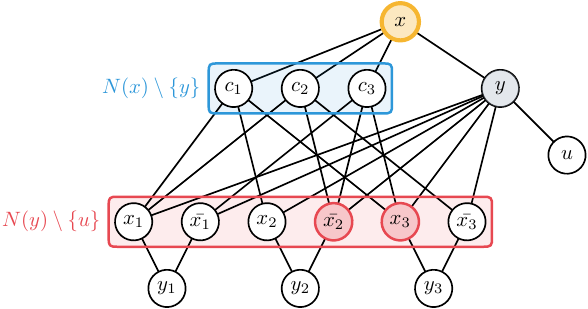}
    \caption{
    The reduction of~\Cref{thm:MRED:method_limitation}.
    Blue zones indicate the vertices representing clauses of \(\phi\), while the red zones represent the literals.
    The shaded vertices form a minimal redundant set where the (yellow) bold vertex is redundant.
    Those in \(N(y) \setminus \{u\}\) correspond to a partial satisfying assignment of $\phi$.
    }
    \label{figure:MRED_bipartite_hardness_construction}
\end{figure}

We turn to the only if part.
Let $R$ be a minimal redundant set of $G$ such that $\red(R) = \{x\}$ and $R \cap N(x) = \{y\}$.
Let $A := R \setminus \{x, y\}$.
By assumption, $N(x) \setminus \{y\}$ is minimally dominated by $A$, and by construction it must thus be that $A \subseteq X$.
As $\red(R) = \{x\}$, $X \subseteq N(y)$, and $N(\ell_i) \setminus \{y_i\} \subseteq N(x)$ for any literal $\ell_i$, any vertex $\ell_i$ must have $y_i$ as private neighbor.
We deduce that at most of one $x_i, \bar{x_i}$ belongs to $A$ for all $i \in [n]$.
This makes $A$ correspond to a partial assignment of the variables $x_1, \dots, x_n$ where the literals in $A$ are set to true.
Since $N(x) \setminus \{y\}$ is dominated by $A$, we deduce that each clause in $\phi$ contains a literal in $A$.
Thus, $A$ corresponds to a partial satisfying assignment of $\phi$ as desired.

Let us now call $\mathcal{R}$ the set of minimal redundant sets $R$ of $G$ that satisfy $\red(R) = \{x\}$ and $R \cap N(x) = \{y\}$ and assume that there exists an output-polynomial time algorithm $\algoA$ listing the members of $\mathcal{R}$ in time $(|G| + |\mathcal{R}|)^c$ for some constant $c$.
We show that $\algoA$ can be used to decide whether $\phi$ is satisfiable in polynomial time.
We build $G$ in time polynomial in the size of $\phi$ then run $\algoA$ with input $G$ for $(|G|+1)^c$ time. 
By previous discussion, if $\algoA$ does not stop within this time bound or stops and outputs a solution, then $\mathcal{R}$ is non-empty and $\phi$ is satisfiable.
If $\algoA$ stops without outputting any solution, then $\phi$ is not satisfiable.
This concludes the proof as the whole procedure runs in polynomial time.
\end{proof}

%Something about: (a) technique fails for MRED; (b) unbounded number of vertices for bipartite; (c) bipartite + one side with bounded degree seems poly delay; (d) ordered generation fails for MIRR; (e) some open problems?

\bibliographystyle{alpha}
\bibliography{main} % unused biblio to remove after

\end{document}